\documentclass[preprint,1p]{elsarticle}
\usepackage[utf8x]{inputenc}
\usepackage{pgfplots}
\usepackage{tikz}
\usepackage{psfrag}

\usepackage{graphicx}        

\usepackage{amssymb}
\usepackage{amsmath}
\usepackage{amsthm}
\usepackage{enumerate}
\usepackage{wrapfig}
\usepackage{subfig}
\usepackage{color}
\usepackage{pstricks}
\usepackage{multicol}
\usepackage{multirow}
\usepackage{url}

\evensidemargin=\oddsidemargin

\newtheorem{Theorem}{Theorem}[section]{\bfseries}{\itshape}
\newtheorem{theorem}[Theorem]{Theorem}{\bfseries}{\itshape}
{\bfseries}{\itshape}
\newtheorem{lemma}[Theorem]{Lemma}{\bfseries}{\itshape}
{\bfseries}{\itshape}
\newtheorem{corollary}[Theorem]{Corollary}{\bfseries}{\itshape}
{\bfseries}{\itshape}
{\bfseries}{\itshape}
\newtheorem{remark}[Theorem]{Remark}{\bfseries}{\itshape}
{\bfseries}{\itshape}
\newtheorem{proposition}[Theorem]{Proposition}{\bfseries}{\itshape}
\newtheorem{assumption}[Theorem]{Assumption}{\bfseries}{\itshape}
\newtheorem{definition}[Theorem]{Definition}{\bfseries}{\itshape}
\newtheorem{example}[Theorem]{Example}{\bfseries}{\itshape}

\def\BbbR{{\mathbb R}}

\def\BbbN{{\mathbb N}}

\newcommand{\eex}{\hbox{}\hfill\rule{0.8ex}{0.8ex}}
\newcommand{\eremk}{\eex}

\newcommand{\M}{{\mathcal M}}
\numberwithin{equation}{section}

\newrgbcolor{DarkOrange}{1 .498 0}
\newrgbcolor{ForestGreen}{0 .6 .4}
\newrgbcolor{Plum}{.4 0 .6}
\newrgbcolor{Brown}{.4 .2 0}
\newif\iftechreport
\techreporttrue
\iftechreport
\else
\fi

\begin{document}
\begin{frontmatter}

\iftechreport 
\title{Robust exponential convergence of $hp$-FEM in balanced norms for singularly perturbed reaction-diffusion problems: corner domains} 
\fi
 \author{M. Faustmann} 
\ead{markus.faustmann@tuwien.ac.at}
 \author{J.M. Melenk\corref{cor1}}
\ead{melenk@tuwien.ac.at}
\address{Technische Universit\"at Wien, Wiedner Hauptstra\ss e 8-10, A-1040 Vienna}


\begin{abstract}
The $hp$-version of the finite element method is applied to singularly perturbed reaction-diffusion type equations on polygonal domains. 
The solution exhibits boundary layers as well as corner layers. On a class of meshes that are suitably refined near the boundary and the corners, 
robust exponential convergence (in the polynomial degree) is shown in both a balanced norm and the maximum norm. 
\end{abstract}
\end{frontmatter}
\section{Introduction}
We consider the boundary value problem 
\begin{equation}
\label{eq:model-problem} 
-\varepsilon^2 \nabla \cdot (A(x) \nabla u)  + c(x) u = f \quad \mbox{ in $\Omega \subset \BbbR^2$}, 
\qquad u|_{\partial\Omega} = 0, 
\end{equation}
where $A$ is pointwise symmetric positive definite and satisfies $A \ge \alpha_0 > 0$ on $\Omega$, and 
$c\ge c_0 >0$ on $\Omega$ for some fixed $\alpha_0$, $c_0 > 0$. 
Furthermore, the functions $A$, $f$, $c$ are assumed to be analytic on $\overline{\Omega}$. 
For the parameter $\varepsilon \in (0,1]$ we focus on the case of small $\varepsilon<\!\!<1$. 
The geometry $\Omega \subset \BbbR^2$ is assumed to be a curvilinear polygon. That is, the boundary $\partial\Omega$
of the bounded Lipschitz domain $\Omega$ consists of finitely many arcs, each of which can be parametrized
by an analytic function. 

The weak formulation of (\ref{eq:model-problem}) is: Find $u \in H^1_0(\Omega)$ such that  
\begin{equation}
\label{eq:weak-formulation}
a(u,v):= \varepsilon^2 \int_\Omega A(x) \nabla u \cdot \nabla v + c(x) u v = \int_\Omega f v 
\qquad \forall v \in H^1_0(\Omega). 
\end{equation}
The bilinear form $a$ induces the energy norm $\|\cdot\|_\varepsilon$ by 
$\|u\|^2_\varepsilon:= a(u,u)$. The Galerkin discretization of (\ref{eq:weak-formulation}) is: 
Given a closed subspace $V_N \subset H^1_0(\Omega)$, find $u_N \in V_N$ such that 
\begin{equation}
\label{eq:abstract-FEM}
a(u_N,v) = \int_\Omega f v \qquad \forall v \in V_N. 
\end{equation}

Clearly, the choice of the space $V_N$ is crucial given that the solution has boundary layers 
near the boundary $\partial\Omega$ and corner singularities at the vertices  of $\Omega$. The boundary layers
can very effectively be captured with anisotropic elements. In the context of the $h$-version FEM, a possibility
are so-called Shishkin meshes as described, for example, in the monograph \cite{roos-stynes-tobiska96}.  
In the context of the $hp$-version FEM, 
\emph{Spectral Boundary Layer Meshes} are appropriate. The latter go back to \cite{schwab-suri96} and 
have been extensively studied in \cite{schwab-suri-xenophontos98,melenk97a,melenk02}. Suitable meshes that also 
resolve the corner singularities are described and analyzed in \cite{melenk02}.  The FEM error $u - u_N$ is naturally analyzed in 
the \emph{energy norm} $\|\cdot\|_\varepsilon$, which is, however, rather weak in the  sense that the energy
norm of the boundary layer contributions tends to zero as $\varepsilon \rightarrow 0$. In particular, a 
convergence analysis that is formulated in the energy norm cannot be expected to yield strong results 
for the error within the layer. This observation has motivated convergence analyses in stronger norms, in particular, 
the so-called \emph{balanced norm}
\begin{equation}
\label{eq:balanced-norm}
\|v\|_{\sqrt{\varepsilon}}:= |v|_{\sqrt{\varepsilon}}+ \|v\|_{L^2(\Omega)}, 
\qquad 
|v|_{\sqrt{\varepsilon}}:= 
\varepsilon^{1/2} |v|_{H^1(\Omega)}. 
\end{equation}
In the $h$-version FEM, robust algebraic convergence of the Galerkin method in this balanced norm has recently
been shown in 1D and 2D for smooth geometries \cite{lin-stynes12,franz-roos14,roos-schopf15}. The corresponding
analysis for the $hp$-version FEM was presented in \cite{melenk-xenophontos16}, where it is shown for 1D and 
2D problems with smooth geometry that robust exponential convergence of the Galerkin method in the balanced 
norm holds true. Robust convergence in the balanced norm is also at the heart of $L^\infty$-estimates
in \cite{melenk-xenophontos16}. In the present work, we extend the analysis of \cite{melenk-xenophontos16} 
to the above setting of piecewise analytic geometries. 
We show robust exponential convergence in the balanced norm (Theorem~\ref{thm:balanced-norm}) 
and the $L^\infty$-norm (Theorem~\ref{thm:Linfty-estimate}). 
As a by-product, the present analysis simplifies some of the 
arguments of \cite{melenk-xenophontos16}. 

\section{An abstract convergence result for the balanced norm}
\label{sec:abstract}
In the following Lemma~\ref{lemma:use-L2-projection} the set $\Omega_0 \subset \Omega$ is an arbitrary open subset. 
In its later application, it will 
be the union of the ``large'' elements of the triangulation;  the complement $\Omega\setminus\Omega_0$ 
will consist of anisotropic elements to capture the boundary layers and of small elements to 
resolve the corner singularities. 
\begin{lemma}
\label{lemma:use-L2-projection}
Let $V_N \subset H^1_0(\Omega)$ be a closed subspace. 
Let $u\in H^1_0(\Omega)$, $u_N \in V_N \subset H^1_0(\Omega)$, and $Iu \in V_N$ satisfy the following orthogonality conditions: 
\begin{align} 
\label{eq:Galerkin-orthogonality}
a(u - u_N,v) &= 0 \qquad \forall v \in V_N, \\ 
\label{eq:L^2-orthogonality-on-Omega_0}
\int_{\Omega_0} c(x) (u - Iu) v &= 0 \qquad \forall v \in V_N|_{\Omega_0}. 
\end{align}
Then, for a constant $C>0$ that depends only on $\|A\|_{L^\infty(\Omega)}$, $\|c\|_{L^\infty(\Omega)}$, $\alpha_0$, and $c_0$: 
\begin{equation}
\label{eq:lemma:use-L2-projection-10}
\varepsilon^{1/2} \|\nabla (u - u_N)\|_{L^2(\Omega)} \leq C 
\left[ \varepsilon^{1/2} \|\nabla( u - Iu)\|_{L^2(\Omega)} + 
\varepsilon^{-1/2} \|u - Iu\|_{L^2(\Omega\setminus\Omega_0)}\right].  
\end{equation}
\end{lemma}
\begin{proof}
We compute 
\begin{align*} 
\|u_N - Iu\|^2_{\varepsilon} &= a(u_N - Iu,u_N - Iu) = a(u - Iu,u_N - Iu) \\
&= \varepsilon^2 \int_\Omega A(x) \nabla( u - Iu) \cdot \nabla (u_N - Iu) + \int_\Omega c(x) (u - Iu) (u_N - Iu)  \\
&= \varepsilon^2 \int_\Omega A(x) \nabla( u - Iu) \cdot \nabla (u_N - Iu) 
+ \int_{\Omega \setminus \Omega_0} c(x) (u - Iu) (u_N - Iu)  \\
&\lesssim \varepsilon^2 \|\nabla(u - Iu)\|_{L^2(\Omega)} \|\nabla( u_N - Iu)\|_{L^2(\Omega)} + 
\|u - Iu\|_{L^2(\Omega\setminus\Omega_0)} \|u_N - Iu\|_{L^2(\Omega\setminus\Omega_0)}. 
\end{align*}
Using Young's inequality we conclude 
\begin{align*}
\|u_N - Iu\|^2_\varepsilon \lesssim \varepsilon^2 \|\nabla (u - Iu)\|^2_{L^2(\Omega)}  + \|u - Iu\|^2_{L^2(\Omega\setminus\Omega_0)}. 
\end{align*}
This implies in particular with the triangle inequality 
\begin{align*}
\varepsilon^{1/2} \|\nabla (u - u_N)\|_{L^2(\Omega)} 
& \leq 
\varepsilon^{1/2} \|\nabla (u - Iu)\|_{L^2(\Omega)}  + 
\varepsilon^{1/2} \|\nabla (u_N - Iu)\|_{L^2(\Omega)}  \\
& \lesssim 
\varepsilon^{1/2} \|\nabla (u - Iu)\|_{L^2(\Omega)}  + \varepsilon^{-1/2} \|u - Iu\|_{L^2(\Omega\setminus\Omega_0)}. 
\qedhere
\end{align*}
\end{proof}
Lemma~\ref{lemma:use-L2-projection} indicates what the ingredients for an analysis in balanced norms are: 
\begin{enumerate}
\item 
The approximation properties of the (weighted) $L^2$-projection $\Pi^{L^2}_{\Omega_0}$ on $\Omega_0$ 
given by 
\begin{equation}
\label{eq:L^2-projection-on-Omega_0} 
\Pi^{L^2}_{\Omega_0} u \in V_N|_{\Omega_0} \qquad \mbox{ s.t. } \qquad 
\int_{\Omega_0} c(x) (u - \Pi^{L^2}_{\Omega_0} u) v = 0\qquad \forall v \in V_N|_{\Omega_0}. 
\end{equation}
\item 
The properties of extension operators ${\mathcal L}_0$ that extend functions from $V_N|_{\Omega_0}$ 
to functions (also in $V_N$) on $\Omega$. Important will be their stability properties measured by  
\begin{equation}
\label{eq:stability-lifting}
\|{\mathcal L}_0\|:= \sup_{v \in V_N \colon v|_{\partial\Omega_0} \ne 0} 
\frac{\varepsilon^{1/2} |{\mathcal L}_0 v|_{H^1(\Omega\setminus\Omega_0)} + 
\varepsilon^{-1/2} \|{\mathcal L}_0 v\|_{L^2(\Omega\setminus\Omega_0)}}{\|v\|_{L^\infty(\partial\Omega_0)}}. 
\end{equation}
\end{enumerate}
\begin{remark} 
While the term 
$\varepsilon^{1/2} |{\mathcal L}_0 v|_{H^1(\Omega\setminus\Omega_0)} + 
\varepsilon^{-1/2} \|{\mathcal L}_0 v\|_{L^2(\Omega\setminus\Omega_0)}$ appears fairly naturally when measuring
the norm of the lifting operator ${\mathcal L}_0$, the denominator $\|v\|_{L^\infty(\partial\Omega_0)}$ is not 
the only ``natural'' choice and could
be replaced with other expressions, e.g., $\|v\|_{L^2(\Omega_0)}$. 
\eremk
\end{remark} 
The operators $\Pi^{L^2}_{\Omega_0}$ and ${\mathcal L}_0$ allow us to formulate an error estimate for the Galerkin
error $u - u_N$ in a balanced norm:
\begin{corollary}
\label{cor:balanced-norm-key-estimate}
In the setting of Lemma~\ref{lemma:use-L2-projection} there exists a constant $C > 0$ 
depending solely on $\|A\|_{L^\infty(\Omega)}$, $\|c\|_{L^\infty(\Omega)}$, $\alpha_0$, $c_0$ such that 
\begin{align}
\label{eq:quasi-optimality-balanced-norm}
|u - u_N|_{\sqrt{\varepsilon}} & \leq C \inf_{v \in V_N} \Bigl[
\|{\mathcal L}_0\| \, \|\Pi^{L^2}_{\Omega_0} u - v\|_{L^\infty(\partial\Omega_0)} \\
\nonumber 
&\phantom{\inf_{v \in V_N} \Bigl[\quad } \mbox{}+ 
\sqrt{\varepsilon} \|\nabla( v - \Pi^{L^2}_{\Omega_0} u)\|_{L^2(\Omega_0)} + 
 |u - v|_{\sqrt{\varepsilon}} + 
\varepsilon^{-1/2} \|u - v\|_{L^2(\Omega\setminus\Omega_0)}
\Bigr].
\end{align}
\end{corollary}
\begin{proof} Let $v \in V_N$. Consider $I u$ defined by 
$$
Iu:= 
\begin{cases}
\Pi^{L^2}_{\Omega_0} u & \quad \mbox{ in $\Omega_0$} \\
v +  {\mathcal L}_0 ((\Pi^{L^2}_{\Omega_0} u)|_{\partial\Omega_0}   - v|_{\partial\Omega_0}) 
& \quad \mbox{ in $\Omega\setminus\Omega_0$}. 
\end{cases}
$$
With the notations 
$|w|_{\sqrt{\varepsilon},\omega} = \sqrt{\varepsilon} \|\nabla w\|_{L^2(\omega)}$ and 
$\|w\|_{\sqrt{\varepsilon},\omega} = \sqrt{\varepsilon} \|\nabla w\|_{L^2(\omega)} + 
 \|w\|_{L^2(\omega)}$ for measurable sets $\omega$, 
Lemma~\ref{lemma:use-L2-projection} implies 
\begin{align*}
|u - u_N|_{\sqrt{\varepsilon}} & \lesssim |u - \Pi^{L^2}_{\Omega_0} u|_{\sqrt{\varepsilon},\Omega_0} + 
|u - Iu|_{\sqrt{\varepsilon},\Omega\setminus\Omega_0} + 
\varepsilon^{-1/2} \|u - Iu\|_{L^2(\Omega\setminus\Omega_0)} \\
&\lesssim 
|u - \Pi^{L^2}_{\Omega_0} u|_{\sqrt{\varepsilon},\Omega_0} + 
|u - v|_{\sqrt{\varepsilon},\Omega\setminus \Omega_0} + 
\varepsilon^{-1/2} \|u - v\|_{L^2(\Omega\setminus\Omega_0)} + 
\|{\mathcal L}_0\| \, \|\Pi^{L^2}_{\Omega_0} u - v\|_{L^\infty(\partial\Omega_0)}. 
\end{align*}
Estimating further with the triangle inequality $|u - \Pi^{L^2}_{\Omega_0} u|_{\sqrt{\varepsilon},\Omega_0} 
\leq 
|u - v |_{\sqrt{\varepsilon},\Omega_0} + |v - \Pi^{L^2}_{\Omega_0} u|_{\sqrt{\varepsilon},\Omega_0} 
$
and then infimizing over $v \in V_N$ gives the result. 
\end{proof}
\section{$hp$-FEM}
\subsection{Spectral Boundary Layer Meshes}
\label{sec:spectral-boundary-layer-mesh}
Rather than considering a general setting, we consider meshes that result from mapping
a few reference configurations. Specifically, we assume that a \emph{fixed} macro-triangulation
${\mathcal T}^{\M} = \{K^{\M} \,|\, K^{\M} \in {\mathcal T}^{\M}\}$ 
consisting of curvilinear quadrilaterals $K^{\M}$ 
with analytic element maps $F_{K^{\M}}:\widehat S:=(0,1)^2 \rightarrow K^{\M}$ 
that satisfy the usual compatibility conditions, i.e., no hanging nodes and if two elements $K_1^{\M}$, $K_2^{\M}$ 
share an edge $e$, then their element maps induce compatible parametrizations
of $e$ (cf., e.g.,  \cite[Def.~{2.4.1}]{melenk02} for the precise conditions). 
Each element of the macro-triangulation may be further refined according to the few
refinement patterns described in Definition~\ref{def:admissible-patterns} (see also \cite[Sec.~{3.3.3}]{melenk02}). 
The actual triangulation
is then obtained by transplanting refinement patterns on the reference square into physical space 
by means of the element maps of the macro-triangulation. That is, the actual element maps are concatenations
of affine maps---which realize the mapping from the reference square or triangle to the elements in the 
refinement pattern---and the element maps of the macro-triangulation. 
\begin{definition}[admissible refinement patterns]
\label{def:admissible-patterns}
Given $\sigma \in (0,1)$, $\kappa \in (0,1/2]$, and $L \in \BbbN_0$, the  following refinement patterns are \emph{admissible}:
\begin{enumerate}
\item 
The \emph{trivial patch}: The reference square $\widehat S$ is not further refined. The corresponding triangulation
of $\widehat S$ consists of the single element: $\widehat{\mathcal T} = \{\widehat S\}$.  
\item
The \emph{boundary layer patch}: $\widehat S $ is split into two elements as depicted 
in Fig.~\ref{fig:bdylayerpatch-convexpatch} (left). We set 
$\widehat{\mathcal T}^{aniso} = \{(0,1) \times (0,\kappa)\}$ and 
$\widehat{\mathcal T}^{large} = \{(0,1) \times (\kappa,1)\}$.  
\item
The \emph{tensor product patch}: $\widehat S$ is split into at least four elements. The small square $(0,\kappa)^2$
may be further refined geometrically with $L \ge 0$ ($L = 0$ corresponds to no refinement) layers
and a geometric grading factor $\sigma \in (0,1)$; see Fig.~\ref{fig:bdylayerpatch-convexpatch} (right). 
The triangulation of $\widehat S$ is decomposed into three types of elements: 
$\widehat{\mathcal T}^{large}$ consists of the one ``large'' element,  
$\widehat{\mathcal T}^{aniso}$ consists of the two anisotropic elements of aspect ratio $\mathcal{O}(1/\kappa)$, 
and $\widehat{\mathcal T}^{CL}$ consists of the elements in $(0,\kappa)^2$. 
\item 
The \emph{mixed patches} and the \emph{geometric patches}: See Fig.~\ref{fig:concavepatch-geometricpatch}.  
The triangulation is decomposed into three types of elements:
$\widehat{\mathcal T}^{large}$ consists of the two ``large'' elements,  
$\widehat{\mathcal T}^{aniso}$ consists of the two anisotropic elements of aspect ratio $\mathcal{O}(1/\kappa)$, 
and $\widehat{\mathcal T}^{CL}$ consists of the elements in $(0,\kappa)^2$. 
\end{enumerate}
\end{definition}
Since our analysis will mostly be done on the reference patterns, we introduce the notations
$\widehat{\mathcal T}^{large}_{K^\M}$,
$\widehat{\mathcal T}^{aniso}_{K^\M}$,
$\widehat{\mathcal T}^{CL}_{K^\M}$ for the sets 
$\widehat{\mathcal T}^{large}$,
$\widehat{\mathcal T}^{aniso}$,
$\widehat{\mathcal T}^{CL}$ that correspond to the chosen refinement pattern for a macro-element
$K^\M \in {\mathcal T}^\M$. 
\begin{remark}
\label{remk:further-refinement-patterns}
The list of admissible refinement patterns in Definition~\ref{def:admissible-patterns} is kept small 
in the interest of simplicity of exposition. A natural extension of the present list of patterns 
includes the case that all quadrilaterals of the refinement pattern are split into two triangles. 
The refinement patterns
employed in the numerical example in Section~\ref{sec:numerics} are also not included in Definition~\ref{def:admissible-patterns} but could be treated with the same techniques. 
In the same vein, the requirement \ref{num:requ} in Definition~\ref{def:bdylayer-mesh} below 
is imposed to shorten notation and to simplify the analysis. 
Furthermore, the stipulation that the macro-triangulation consist of quadrilaterals only could be relaxed to 
include triangles as well. Nevertheless, we point out that any triangulation consisting of triangles can 
be turned into one consisting of quadrilaterals by a suitable uniform refinement. 
\eremk
\end{remark}
We consider meshes  that result from applying these refinement patterns. Additionally, 
we impose further conditions on the choice of the refinement patterns for each element of the macro-triangulation: 
\begin{definition}[spectral boundary layer mesh ${\mathcal T}(\kappa,{\mathbf L})$]
\label{def:bdylayer-mesh}
Let ${\mathcal T}^{\M}$ be a fixed macro-triangulation consisting of quadrilaterals with analytic element maps 
that satisfies \cite[Def.~{2.4.1}]{melenk02}. Fix $\sigma \in (0,1)$ and $\kappa \in (0,1/2]$. 
For each macro-element $K^{\M}$, select $L_{K^{\M}} \in \BbbN_0$ and write ${\mathbf L}=  (L_{K^{\M}})_{K^{\M} \in {\mathcal T}^{\M}}$. 
A mesh ${\mathcal T}(\kappa,{\mathbf L})$ 
is called a \emph{spectral boundary layer mesh} if the following conditions are satisfied: 
\begin{enumerate}
\item ${\mathcal T}(\kappa,{\mathbf L})$ is obtained by refining each element $K^{\M} \in {\mathcal T}^{\M}$ 
according to one of the refinement patterns given in Definition~\ref{def:admissible-patterns} using 
the parameters $\sigma$, $\kappa$, and $L_{K^{\M}}$. 
\item The resulting mesh ${\mathcal T}(\kappa,{\mathbf L})$ is a regular triangulation of $\Omega$, i.e., it does not have hanging nodes. 
Since the element maps for the refinement patterns 
are assumed to be affine, this requirement ensures that the resulting triangulation satisfies
\cite[Def.~{2.4.1}]{melenk02}. 
\item 
\label{num:requ}
For each element $K^{\M}$ of the macro triangulation, we assume that for 
the intersection $\overline{K^{\M}}\cap \partial \Omega$ only the following cases can arise: 
a) it is empty; b) it consists of exactly  one vertex; c) it consists of the closure of exactly one edge; 
d) it consists of the closure of exactly two edges intersecting at a corner of $\partial \Omega$. 
\end{enumerate}
Further conditions on the choice of the refinement patterns are as follows:
\begin{enumerate}
\setcounter{enumi}{3}
\item If exactly one edge $e$ of a macro-element $K^{\M} \in {\mathcal T}^{\M}$  lies on $\partial\Omega$, then 
the corresponding reference edge $\widehat e = F^{-1}_{K^\M}(e)$ is $(0,1) \times \{0\}$, and the refinement patterns
of the boundary layer patch or the mixed patch is applied.
\item If exactly two edges $e_1$, $e_2$ of a macro-element $K^{\M} \in {\mathcal T}^{\M}$ lie on $\partial\Omega$, 
then the corresponding reference edges are $(0,1) \times  \{0\}$ and $\{0\} \times (0,1)$, and the tensor product 
refinement pattern is applied.
\item If exactly one vertex of a macro-element $K^{\M} \in {\mathcal T}^{\M}$ lies on $\partial\Omega$, then 
the corresponding reference vertex is the point $(0,0)$ and the refinement pattern is either the tensor product
patch, the mixed patch, or the geometric patch. 
\end{enumerate}
\end{definition}
\begin{figure}
\psfragscanon
\quad
\includegraphics[width=0.4\textwidth]{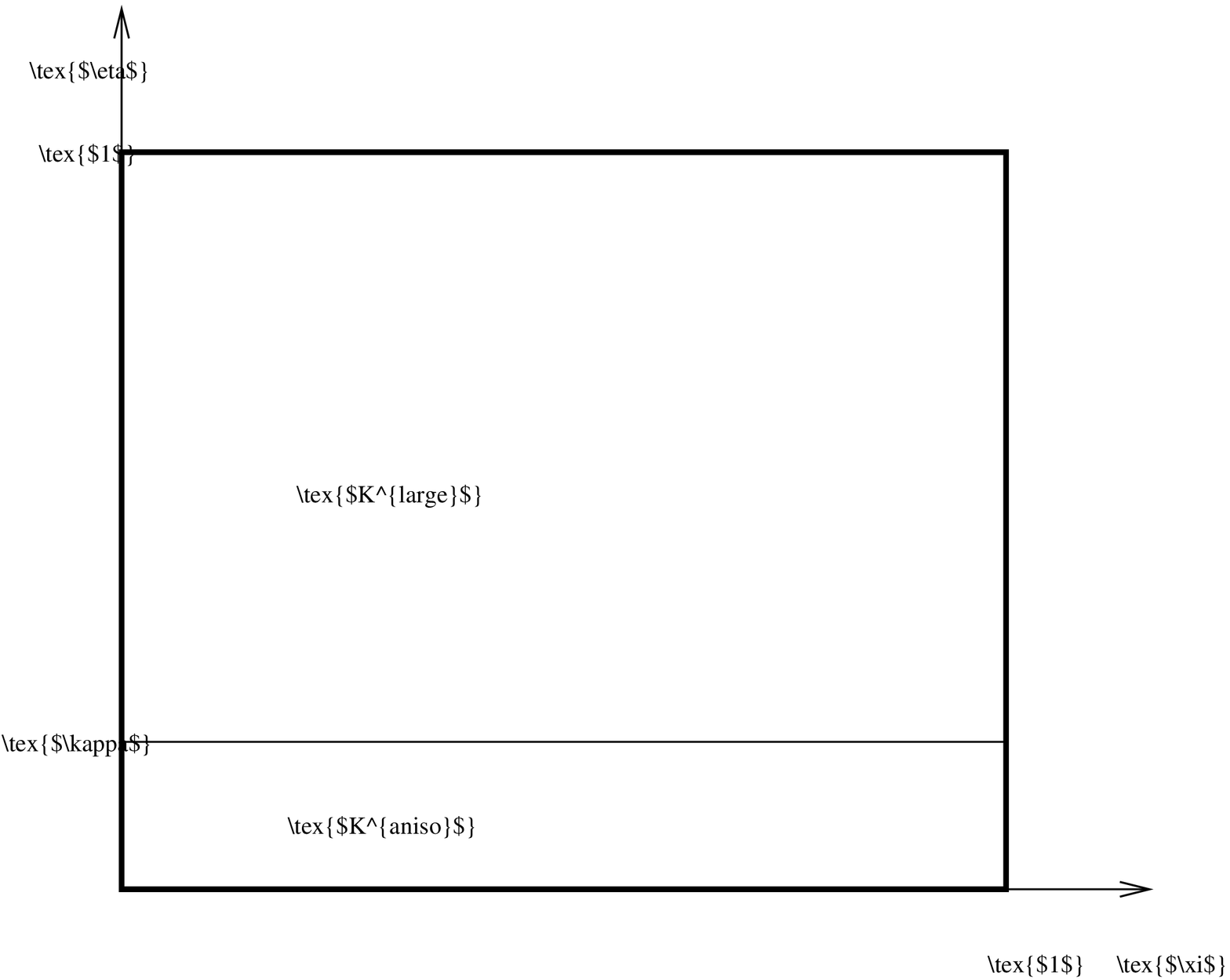}
\hfill
\includegraphics[width=0.5\textwidth]{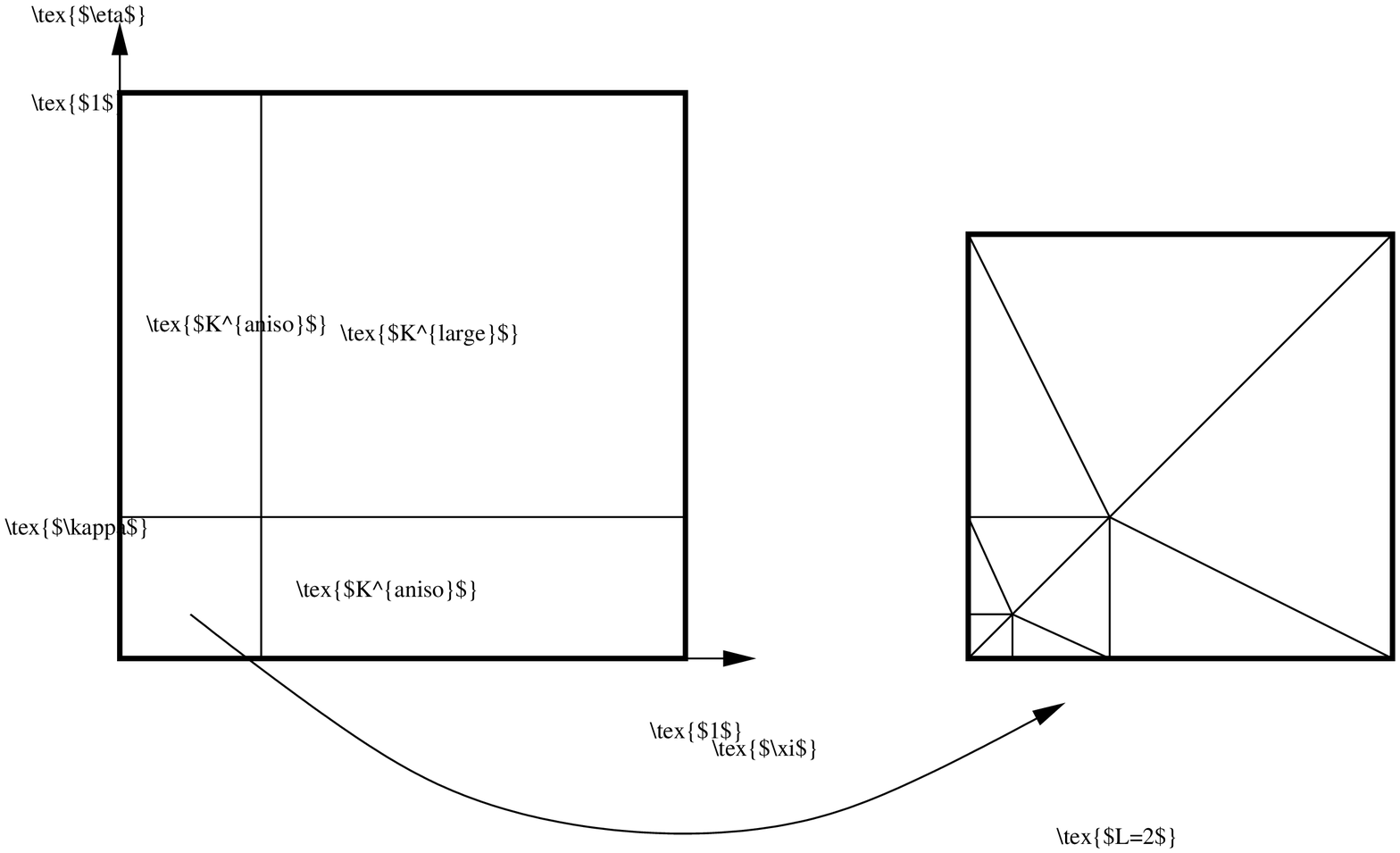}
\quad
\psfragscanoff
\caption{
\label{fig:bdylayerpatch-convexpatch} Reference boundary layer patch (left); reference tensor product patch (right).}
\end{figure}
\begin{figure}
\psfragscanon
\quad
\includegraphics[width=0.45\textwidth]{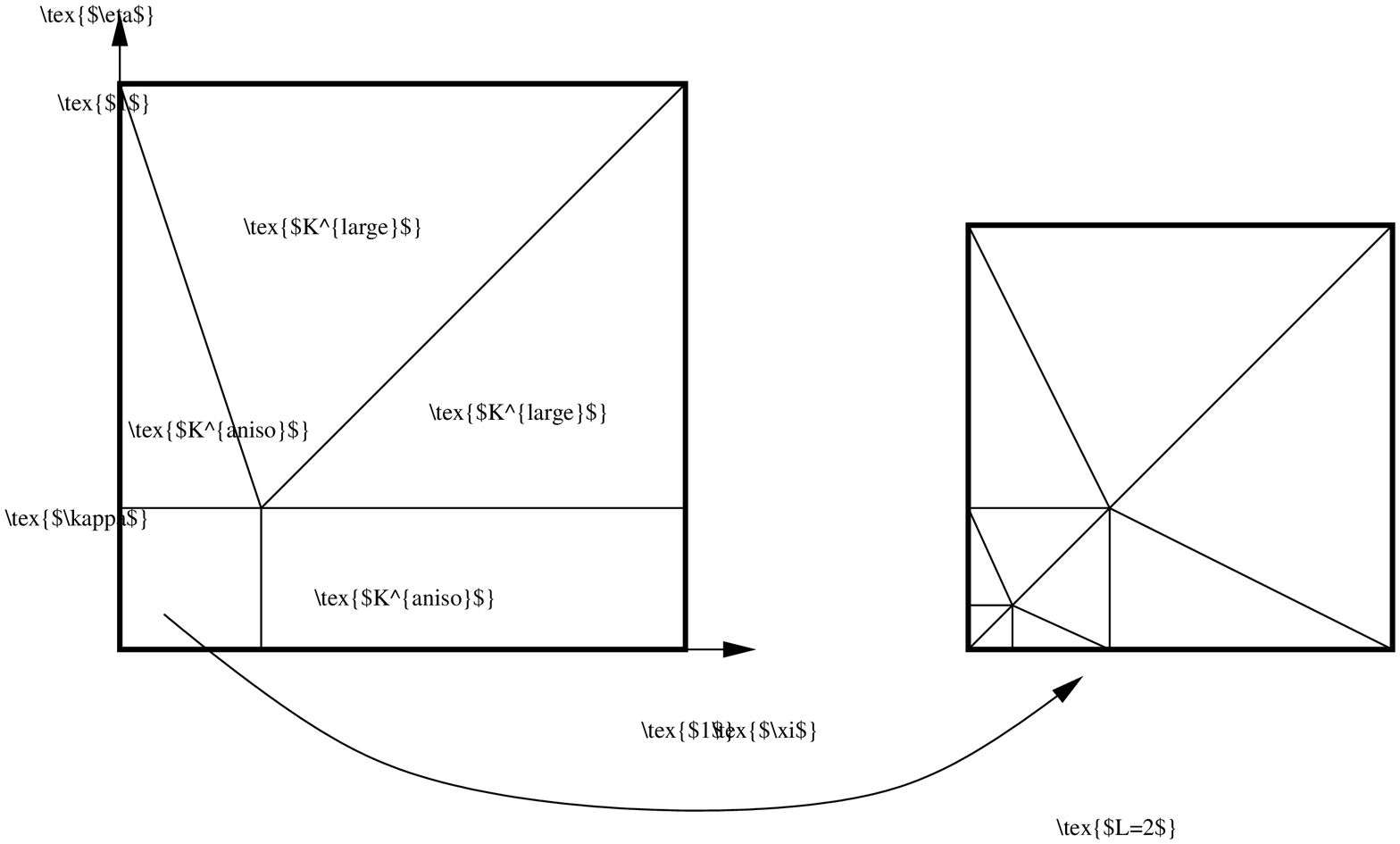}
\hfill
\includegraphics[width=0.45\textwidth]{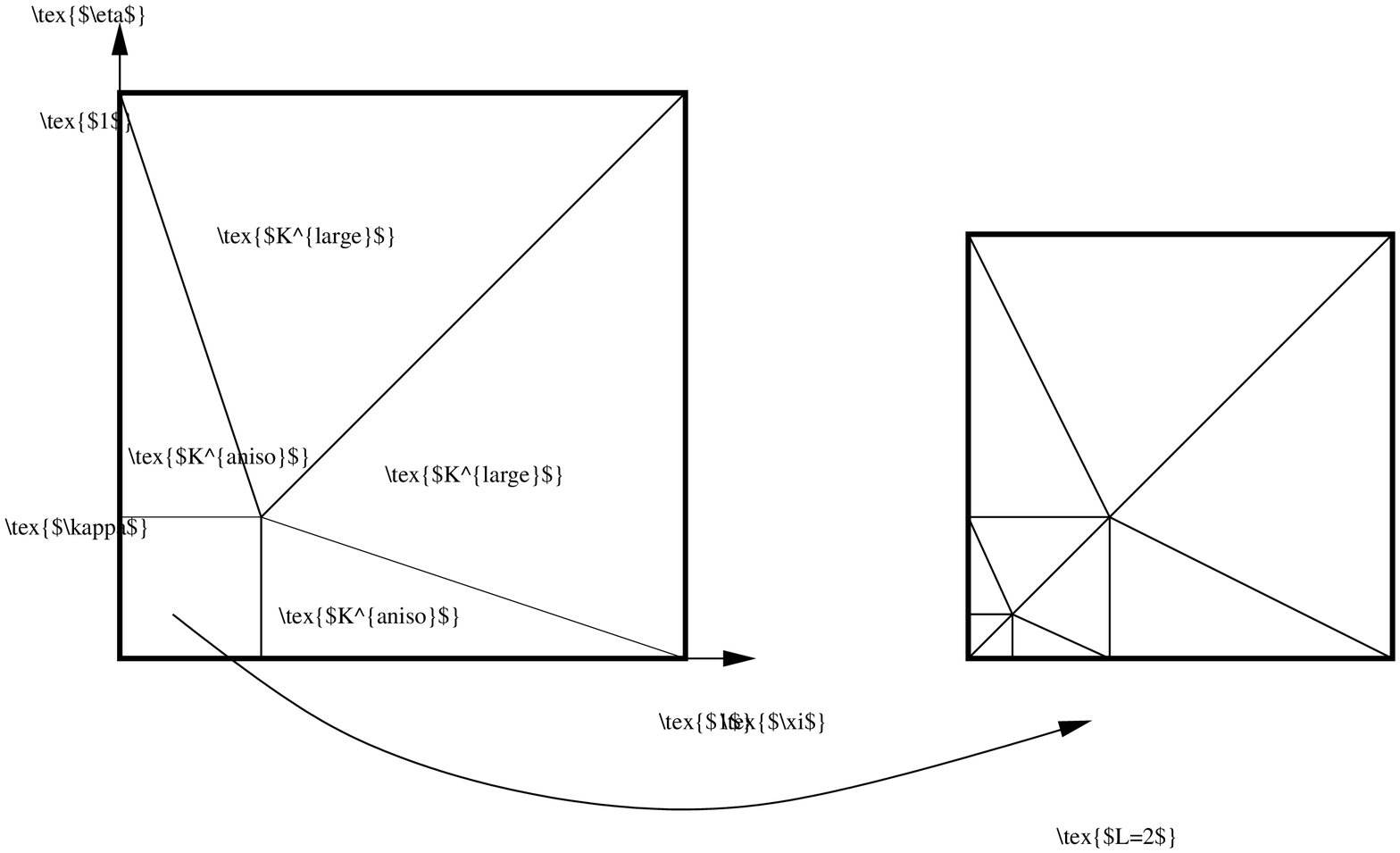}
\quad
\psfragscanoff
\caption{\label{fig:concavepatch-geometricpatch} Reference mixed patch (left); reference geometric patch (right).}
\end{figure}
We refer the reader to Fig.~\ref{fig:notation_lifting} (right) for an example of a spectral boundary layer mesh. 

The element maps $F_K:\widehat K \rightarrow K$, $K \in {\mathcal T}(\kappa,{\mathbf L})$, are given by 
the concatenation of affine maps from the reference square or triangle and the  patch maps $F_{K^{\M}}$.  
Here, the reference square is $(0,1)^2$ and the reference triangle $\{(x,y)\,|\, 0 < x < 1, 0 < y < 1-x\}$. 
For any triangulation ${\mathcal T}$ we define 
$S^{p,1}({\mathcal T}):= \{u \in H^1(\Omega)\,|\, u|_K \circ F_K \in \Pi_p(\widehat K)\}$, where $\Pi_p(\widehat K)$ 
is the tensor product space ${\mathcal Q}_p = \operatorname*{span} \{x^iy^j\,|\, 0 \leq i, j \leq p\}$ if 
$\widehat K = (0,1)^2$ and 
the space ${\mathcal P}_p = \operatorname*{span} \{x^i y^j\,|\, 0 \leq i+j \leq p\}$ if $\widehat K$ is the 
reference triangle. Finally, we set $S^{p,1}_0({\mathcal T}):= S^{p,1}({\mathcal T}) \cap H^1_0(\Omega)$. 

A triangulation ${\mathcal T}(\kappa,{\mathbf L})$ has three types of elements that cover the regions 
$\Omega_0$, $\Omega_{aniso}$, and $\Omega_{CL}$: 
\begin{definition}[$\Omega_0$, $\Omega_{aniso}$, $\Omega_{CL}$, ${\mathcal V}_{CL}$]
\begin{enumerate}
\item 
The ``large'' elements ${\mathcal T}^{large}$. These are the images (under the macro-element maps) 
of the trivial patch or the large 
elements (denoted $K^{large}$ in Figs.~\ref{fig:bdylayerpatch-convexpatch}, \ref{fig:concavepatch-geometricpatch}). 
These elements are  shape regular.  
We set $\Omega_0:=\left(\cup \{\overline{K_i}\,|\, K_i \in {\mathcal T}^{large}\}\right)^\circ$
\item 
The ``anisotropic elements'' ${\mathcal T}^{aniso}$. These elements are the images (under the macro-element maps) 
of elements of aspect ratio $\mathcal{O}(1/\kappa)$, which are denoted by $K^{aniso}$ in
Figs.~\ref{fig:bdylayerpatch-convexpatch}, \ref{fig:concavepatch-geometricpatch}. 
We set $\Omega_{aniso}:=\left(\cup \{\overline{K_i}\,|\, K_i \in {\mathcal T}^{aniso}\}\right)^\circ$. 
\item
The ``corner layer elements'' ${\mathcal T}^{CL}$:  
These elements are the images of elements in the $\mathcal{O}(\kappa)$-neighborhood 
of $(0,0)$ of the reference pattern. These elements are shape regular. 
We set $\Omega_{CL} = \left(\cup \{\overline{K_i}\,|\, K_i \in {\mathcal T}^{CL}\}\right)^{\circ}$. 
\end{enumerate} 
${\mathcal V}_{CL}:=\{F_{K^\M}(0,0)\,|\, \mbox{$K^\M $ is either a tensor product or a mixed patch or geometric patch}\}$ denotes 
the set of vertices of the macro-triangulation towards which potentially geometric refinement is done.
\end{definition}
\begin{remark} 
Key properties of the meshes ${\mathcal T}(\kappa,{\mathbf L})$ are: a) the elements abutting $\partial\Omega$ 
are either anisotropic or from $\Omega_{CL}$; b) geometric refinement can be ensured near the vertices of $\Omega$; 
c) there is $\mu > 0$ (depending only on the macro-triangulation) such that 
\begin{equation}
\label{eq:Omega0-away-from-partialOmega}
\operatorname*{dist}(\Omega_0,\partial\Omega) \ge \mu \kappa. 
\end{equation}
In the notation of \cite[Sec.~{3.3.2}]{melenk02}, these meshes are patchwise structured meshes. 
In particular, therefore, the piecewise polynomial spaces $S^{p,1}_0({\mathcal T}(\kappa,{\mathbf L}))$ have approximation 
properties that were analyzed in \cite[Sec.~{3.4.2}]{melenk02} and discussed in more detail in Proposition~\ref{prop:thm.3.4.8} below. 
\eremk
\end{remark}
\begin{remark}
(Scaling arguments) In our analysis, we will frequently appeal to scaling arguments. Strictly speaking, such
arguments apply only to affine element maps. In the present case, the element maps are concatenations
of affine maps and a fixed number of analytic diffeomorphisms (given by the element maps of the macro-triangulation). 
Hence, scaling arguments can be brought to bear for elements of the reference patterns on $\widehat S$ and then 
transplanted with the macro-element maps. In effect, therefore, scaling argument can be applied for estimates 
in $L^2$ and the $H^1$-seminorm. 
\eremk
\end{remark}
\begin{figure}
\psfragscanon
\quad
\includegraphics[width=0.5\textwidth]{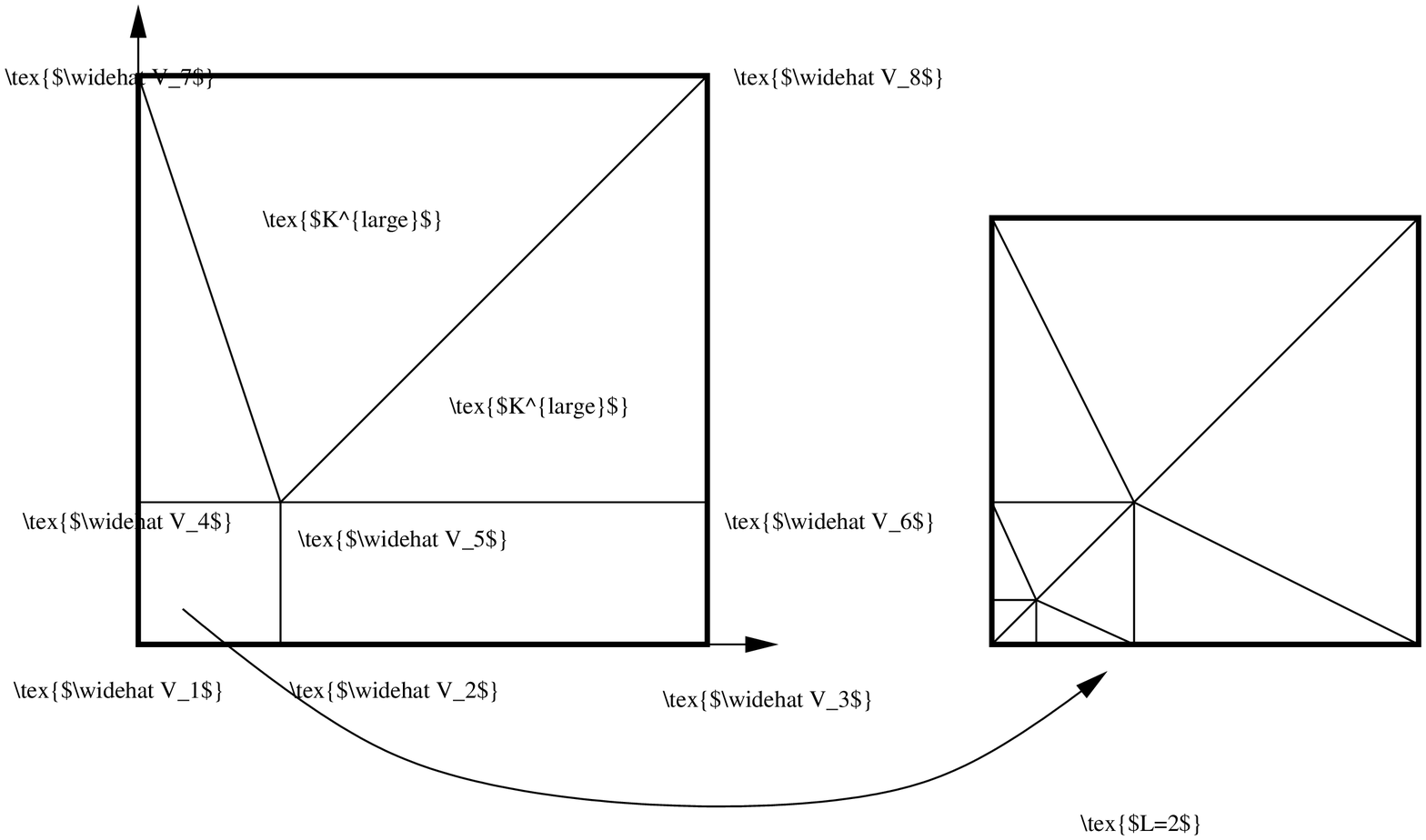}
\hfill \raisebox{10mm}{ \includegraphics[width=0.35\textwidth]{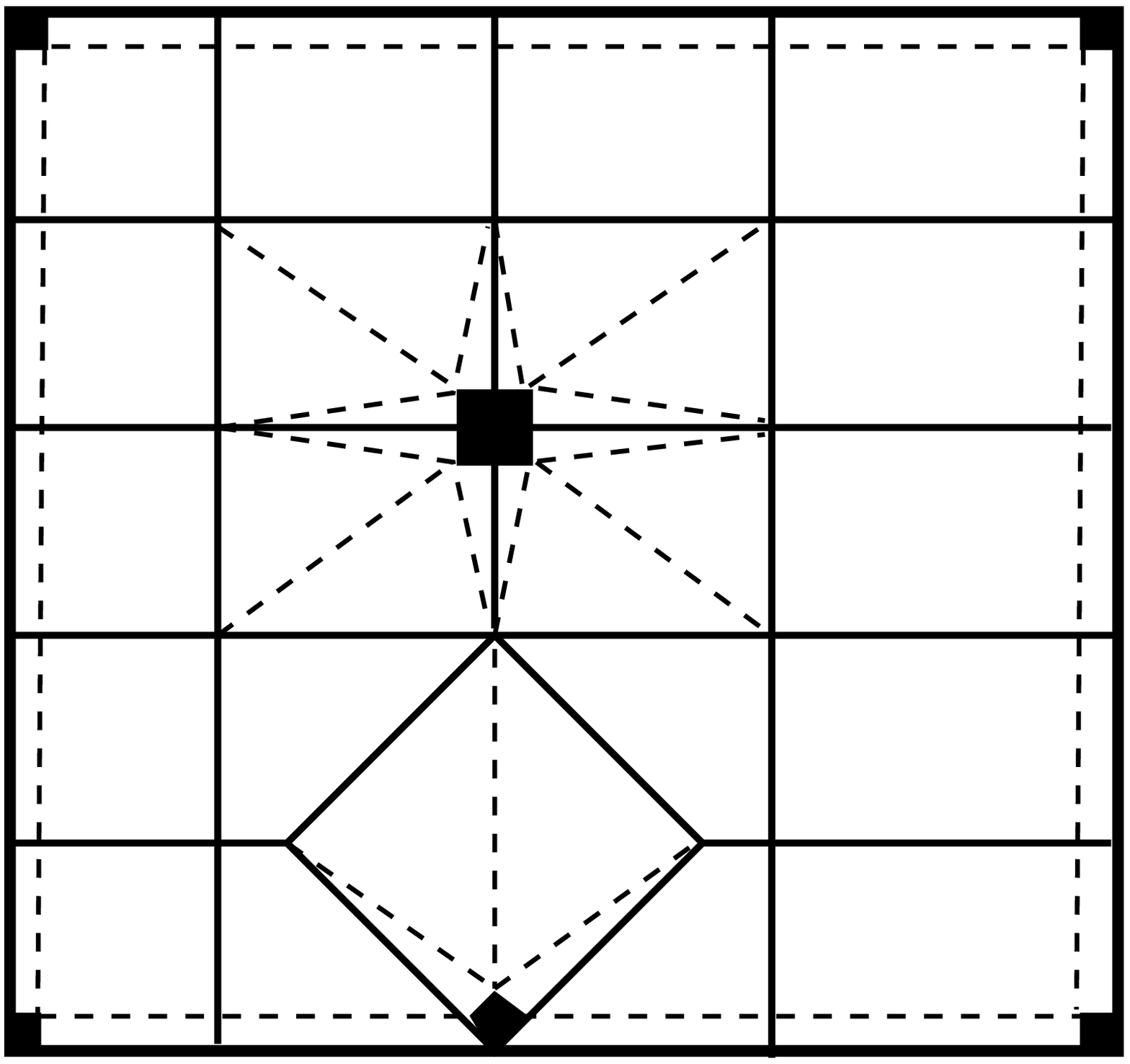}}
\psfragscanoff
\caption{
\label{fig:notation_lifting} 
Left: Notation for the construction of the lifting ${\mathcal L}_0$. Right: Example of a spectral boundary layer mesh. Solid lines indicate macro elements, dashed ones mesh lines of the refinement patterns, solid regions contain geometric refinement.}  
\end{figure}
The following lemma constructs a lifting ${\mathcal L}_0$ as required in Section~\ref{sec:abstract}. 
${\mathcal L}_0$ is controlled in the stronger norm $\|\cdot\|_+$ given by  
\begin{equation}
\label{eq:norm-mathcalL}
\|{\mathcal L}_0\| \leq 
\|{\mathcal L}_0\|_+:= \sup_{v \in V_N \colon v|_{\partial\Omega_0} \ne 0} 
\frac{\| {\mathcal L}_0 v\|_{L^\infty(\Omega)} + \varepsilon^{1/2} |{\mathcal L}_0 v|_{H^1(\Omega\setminus\Omega_0)} + 
\varepsilon^{-1/2} \|{\mathcal L}_0 v\|_{L^2(\Omega\setminus\Omega_0)}}{\|v\|_{L^\infty(\partial\Omega_0)}}. 
\end{equation}
\begin{lemma}
\label{lemma:L0} 
Let ${\mathcal T}(\kappa,{\mathbf L})$ be a spectral boundary layer mesh (Def.~\ref{def:bdylayer-mesh}). 
Set $V_N:= S^{p,1}_0({\mathcal T}(\kappa,{\mathbf L}))$. There exists a lifting operator 
${\mathcal L}_0: V_N|_{\partial\Omega_0} \rightarrow V_N$ with 
\begin{align}
\|{\mathcal L}_0\| \leq 
\|{\mathcal L}_0\|_+
& \leq C  \left( p^2 \frac{\varepsilon^{1/2}}{{\kappa^{1/2}}}  + \frac{\kappa^{1/2}}{\varepsilon^{1/2}}\right). 
\end{align}
The constant $C > 0$ depends only on the shape-regularity of the macro-triangulation ${\mathcal T}^{\M}$. 
\end{lemma}
\begin{proof}
The lifting ${\mathcal L}_0 u$ is constructed patchwise, i.e., for each $K^{\M} \in {\mathcal T}^{\M}$
separately. To fix ideas, we construct the lifting for one refinement pattern only, namely, for 
a macro-element $K^{\M}$ that corresponds to the mixed patch of 
Fig.~\ref{fig:concavepatch-geometricpatch} (left). We use the vertices $\widehat V_i$, $i=1,\ldots,8$ 
as shown in Fig.~\ref{fig:notation_lifting} (left). 
It is convenient to introduce 
$\partial \widehat \Omega_0:= F_{K^{\M}}^{-1}(\partial\Omega_0 \cap \overline{K^{\M}})$. 
This set consists of the union of the closure of edges and possibly single points. 
For the present case of the mixed patch it contains at least the two edges 
$\widehat e_1:= (\widehat V_5, \widehat V_6)$ and $\widehat e_2:= (\widehat V_5, \widehat V_7)$. 
Let $u \in V_N|_{\partial{\Omega_0}}$.  We denote by 
$\widehat u$ the pull-back of $u$ to the reference configuration, i.e., $\widehat u:= u \circ F_{K^{\M}}$, 
which is defined on $\partial\widehat\Omega_0$. 
We now proceed to define the lifting 
${\mathcal L}_0 u$ by prescribing its extension $\widehat{\mathcal L}_0 \widehat u$ on the reference 
patch. This is achieved in two steps: first, the values of $\widehat{\mathcal L}_0 \widehat u$ 
are fixed on the edges of the macro-element; in a second step, the values of the edges are lifted 
to the elements. Let $V_i:=F_{K^\M}(\widehat V_i)$ denote the image of $\widehat V_i$ under the patch map.  
We define $\widehat{\mathcal L}_0 \widehat u$ in the nodes $\widehat V_i$, $i=1,\ldots,7$, 
as follows: If $V_i \in \partial{\Omega_0}$, then 
$(\widehat {\mathcal L}_0 \widehat u)(\widehat V_i) := \widehat u(\widehat V_i) = u(V_i)$. 
If $V_i \not\in \partial{\Omega_0}$, then $(\widehat {\mathcal L}_0 \widehat u)(\widehat V_i):= 0$.  
Next, we define the values of $\widehat{\mathcal L}_0 \widehat u$ on the $9$ edges given by 
$(\widehat V_1,\widehat V_2)$, 
$(\widehat V_2,\widehat V_3)$, 
$(\widehat V_1,\widehat V_4)$, 
$(\widehat V_4,\widehat V_7)$, 
$(\widehat V_2,\widehat V_5)$, 
$(\widehat V_4,\widehat V_5)$,
$(\widehat V_3,\widehat V_6)$, 
$(\widehat V_6,\widehat V_8)$, 
$(\widehat V_7,\widehat V_8)$. 
If the push-forward of an edge lies in $\partial{\Omega_0}$, then 
$\widehat{\mathcal L}_0\widehat u$ is already defined by $\widehat u$. If the push-forward does not lie 
in $\partial{\Omega_0}$, then we let $\widehat{\mathcal L}_0 \widehat u$ be the linear interpolant between
the values at the endpoints (which we have defined above already). 
This defines $\widehat{\mathcal L}_0 \widehat u$ in particular on the 
boundary $\partial (0,1)^2$ of the reference refinement pattern. Consider the case $L = 0$, i.e., 
the square $(0,\kappa)^2$ is not further refined. Then $\widehat{\mathcal L}_0 \widehat u$ is already determined 
on all edges of the reference patch, and we can lift from the edges to the elements on which 
$\widehat{\mathcal L}_0 \widehat u$ is still undefined. A fairly standard lifting, which we describe in 
Lemma~\ref{lemma:lifting} for the reader's convenience, ensures  
\begin{equation}
\label{eq:gradient}
\| \widehat{\mathcal L}_0 \widehat u\|_{L^\infty(\widehat S)} \leq 
C\max_{(x,y) \in \partial \widehat\Omega_0}  | \widehat u(x,y)|, 
\qquad 
\|\nabla  \widehat{\mathcal L}_0 \widehat u\|_{L^\infty(\widehat S)} \leq C \kappa^{-1} p^2 
\max_{(x,y) \in \partial\widehat\Omega_0} | \widehat u(x,y)|. 
\end{equation}
Here, the factor $\kappa^{-1}$ appears since the lifting is done on elements of aspect ratio 
$\mathcal {O}(1/\kappa)$. 
We note that the thus defined function $\widehat{\mathcal L}_0 \widehat u$ is also a continuous, piecewise polynomial
of degree $p$ if the square $(0,\kappa)^2$ is refined geometrically with $L > 0$ layers. 
Thus, we have constructed a lifting.  Noting that $F_{K^{\M}}^{-1}(K^{\M} \setminus\Omega_0) 
\subset (0,1)^2 \setminus (\kappa,1)^2$ for the mixed patch, we infer
\begin{align*}
\| {\mathcal L}_0 u\|_{L^2(K^{\M}\setminus\Omega_0)} & \leq C \sqrt{\kappa} 
\max_{(x,y) \in \overline{K^\M} \cap \partial\Omega_0} | u(x,y)|,  \\
\|\nabla  {\mathcal L}_0  u\|_{L^2(K^{\M}\setminus\Omega_0)} & \leq C \kappa^{-1/2} p^2 
\max_{(x,y) \in \overline{K^\M} \cap \partial\Omega_0} | u(x,y)|. 
\end{align*}
In this way, we construct the lifting for each refinement pattern and consequently patch by patch. It is essential to note that our assumptions 
on the refinement patterns are such that the patchwise defined lifting is continuous across patch boundaries, i.e., 
it is actually in 
$S^{p,1}_0({\mathcal T}(\kappa,{\mathbf L}))$. 
\end{proof}
We will also need a second lifting operator: 
\begin{lemma}
\label{lemma:widetilde_L0}
There is a lifting ${\mathcal L}_{CL}: V_N|_{\partial\Omega_{CL}} \rightarrow V_N$
with the following property: 
\begin{align}
\label{eq:lemma:widetilde_L0-10}
\|{\mathcal L}_{CL}\| := \!\! \sup_{v \in V_N: v|_{\partial\Omega_{CL}}\neq 0} \!\!
\frac{|{\mathcal L}_{CL} v|_{H^1(\Omega_{CL})} + \varepsilon^{-1} \|{\mathcal L}_{CL} v\|_{L^2(\Omega_{CL})}+ \|{\mathcal L}_{CL} v\|_{L^{\infty}(\Omega_{CL})}} {\|v\|_{L^\infty(\partial\Omega_{CL})} }
\leq C \left[ p^2 + \frac{\kappa}{\varepsilon} \right].
\end{align}
\end{lemma}
\begin{proof}
The lifting is again constructed patchwise. For simplicity, we will not construct the lifting to $\Omega$ 
but only to $\Omega_{CL}$, since we are only interested in $({\mathcal L}_{CL} u)|_{\Omega_{CL}}$. 

We observe that $\Omega_{CL}$ is the union of images of the square $(0,\kappa)^2$ under certain macro-element maps
and that push-forwards of the lines $\widehat e_1:= (0,\kappa) \times \{\kappa\}$ 
and $\widehat e_2:= \{\kappa\} \times (0,\kappa)$ form part 
of the boundary of $\Omega_{CL}$. The remaining two lines $\widehat e_3:= (0,\kappa) \times \{0\}$ 
and $\widehat e_4:= \{0\} \times  (0,\kappa)$
are either mapped to subsets  of $\partial\Omega$ or are meshlines that are completely inside $\Omega_{CL}$. 
On $(0,\kappa)^2$ we define ${\mathcal L}_{CL}$ as follows: Let $u \in V_N|_{\partial\Omega_{CL}}$ 
and $\widehat u$ be its pull-back under the macro-element map $F_{K^\M}$. 
Fix $\widehat {\mathcal L}_{CL} \widehat u$ to coincide with $\widehat u$ on the lines $\widehat e_1$ and $\widehat e_2$, 
to be zero in $(0,0)$, and to be the linear interpolant on the 
remaining two edges $\widehat e_3$ and $\widehat e_4$. 
Finally, $\widehat{\mathcal L}_{CL} \widehat u$ is lifted to $(0,\kappa)^2$ by a standard lifting, e.g., the one 
constructed in Lemma~\ref{lemma:lifting}. 
We conclude 
\begin{equation}
\label{eq:lemma:widetilde_L0-50}
\|\widehat {\mathcal L}_{CL} \widehat u\|_{L^\infty((0,\kappa)^2)} + p^{-2} \kappa
\|\nabla \widehat {\mathcal L}_{CL} \widehat u\|_{L^\infty((0,\kappa)^2)} 
\lesssim \|\widehat u\|_{L^\infty(\widehat e_1 \cup \widehat e_2)}. 
\end{equation}
Transforming to $F_{K^\M}((0,\kappa)^2)$ yields (\ref{eq:lemma:widetilde_L0-10}). 
\end{proof}
By selecting $\kappa = {\mathcal O}(p\varepsilon)$ in the spectral boundary layer meshes ${\mathcal T}(\kappa,{\mathbf L})$, 
one can construct an approximation $\Pi u \in S^{p,1}_0({\mathcal T}(\kappa,{\mathbf L}))$ with the following approximation properties: 
\begin{proposition}
\label{prop:thm.3.4.8} 
For parameters $\lambda > 0$ and polynomial degrees $p$ consider meshes 
${\mathcal T}(\min\{\lambda p \varepsilon,1/2\},{\mathbf L})$. 
Let $L$ be defined as $L:=\min\{{L}_{K^\M}\,|\, K^{\M}\,\text{s.t.}\, F_{K^\M}(0,0)\,\text{is a vertex of }\, \Omega\}$.
Then, there exist $\lambda_0 > 0$, $b$, $C > 0$ 
depending only on $\Omega$, $A$, $c$, and $f$ such that the 
following is true: For every $\lambda \in (0,\lambda_0]$  
there exists an approximation 
$\Pi u \in S^{p,1}_0({\mathcal T}(\min\{\lambda p \varepsilon,1/2\},{\mathbf L}))$ such that 
\begin{subequations}
\begin{align}
\label{prop:thm.3.4.8-i}
\|u - \Pi u\|_{L^\infty(\Omega)} 
&\leq C p^4 \left( \lambda^{-1/2} e^{-b \lambda p} +  p e^{-bL} \right), \\
\label{prop:thm.3.4.8-ia}
\|u - \Pi u\|_{L^\infty(\Omega\setminus\Omega_{CL})} 
&\leq C p^4  \lambda^{-1/2} e^{-b \lambda p}, \\
\label{prop:thm.3.4.8-ii}
\varepsilon^{1/2} \|\nabla( u - \Pi u)\|_{L^2(\Omega)} &\leq C p^4 \left(  \lambda^{-1/2} e^{- b\lambda  p} 
+ \sqrt{\varepsilon} p^2 e^{-bL}\right). 
\end{align}
\end{subequations}
Furthermore, for arbitrary $\overline{x} \in \BbbR^2$ we have
\begin{align}
\|\nabla (u - \Pi u)\|_{L^2(B_{\lambda p \varepsilon}(\overline{x})\cap\Omega_{CL})} 
&\leq C p^4 \left( \lambda^{-1/2} e^{-b \lambda p} + p^2 e^{-bL} \right), \\
\|\nabla (u - \Pi u)\|_{L^2(B_{\lambda p \varepsilon}(\overline{x})\cap(\Omega\backslash\Omega_{CL}))} 
&\leq C p^4 \lambda^{-1/2} e^{-b \lambda p}. 
\end{align}
\end{proposition}
\begin{proof} 
The result relies on a careful inspection of \cite[Thm.~{3.4.8}]{melenk02} and a modification
of the boundary layer approximation that goes back to \cite{schwab-suri96} for the 1D case (see also 
\cite{melenk-xenophontos16}). Of interest to us and in our tracking of the procedure in 
\cite[Thm.~{3.4.8}]{melenk02} is the case that $\lambda p \varepsilon$ is sufficiently small, since in
 the case $\lambda p \varepsilon \gtrsim 1$, we may estimate $\varepsilon^{-1} \lesssim \lambda p$ and absorb
powers of $\lambda p$ in the exponentially decaying term $e^{-b \lambda p}$. It is also that regime
that is responsible for the factor $\sqrt{\varepsilon} p^2 e^{-bL}$ instead of a factor 
$\sqrt{\varepsilon} p e^{-bL}$ in (\ref{prop:thm.3.4.8-ii}). 
We emphasize that in the course of the proof, the constants $C$, $b>0$ may be different in each occurrence. 
We also mention that the factor $p^4$ results from simply estimating  $p^3 (1 + \ln p) 
\lesssim p^4$. 

Let $\Pi^\infty_{p,{\mathcal T}}$ be the approximation operator employed in the proof of 
\cite[Thm.~{3.4.8}]{melenk02}.  This operator has the following stability properties by 
\cite[Thm.~{3.2.20}]{melenk02}: Its pull-back $\Pi^\infty_{p,\widehat K}$ 
to the reference element $\widehat K$ (which can be either 
the reference triangle or the reference square) satisfies 
\begin{align}
\label{eq:Pi_infty-1}
& (\Pi^\infty_{p,\widehat K} v)|_e 
\quad \mbox{ coincides with the Gauss-Lobatto interpolant of $v|_e$ on each edge $e$ of 
$\widehat K$}, \\
\label{eq:Pi_infty-2}
&\|v - \Pi^\infty_{p,\widehat K} v\|_{L^\infty(\widehat K)} 
\lesssim C_p \inf_{w \in \Pi_p(\widehat K)} \|v - w\|_{L^\infty(\widehat K)}, 
\qquad C_p:= p (1 + \ln p), \\
\label{eq:Pi_infty-3}
&\|\nabla (v - \Pi^\infty_{p,\widehat K} v)\|_{L^2(\widehat K)} 
\lesssim \inf_{w \in \Pi_p(\widehat K)} \|\nabla(v - w)\|_{L^2(\widehat K)} + p^2 C_p \|v - w\|_{L^\infty(\widehat K)}.
\end{align}
We inspect the proof of \cite[Thm.~{3.4.8}]{melenk02}, which studies $u - \Pi^\infty_{p,{\mathcal T}} u$, and 
modify as needed. The exact solution $u$ of (\ref{eq:model-problem}) is written as 
$u = w_\varepsilon + \chi^{BL} u^{BL}_\varepsilon + \chi^{CL} u^{CL}_\varepsilon 
+ r_\varepsilon$, where $w_\varepsilon$ represents
the smooth part of an asymptotic expansion, $u^{BL}_\varepsilon$ 
the boundary layer part, $u^{CL}_\varepsilon$ the corner layer
and $r_\varepsilon$ the (exponentially small) remainder; the smooth cut-off functions $\chi^{BL}$, $\chi^{CL}$
localize near $\partial\Omega$ and the vertices of $\Omega$, respectively. 
(The properties of $w_\varepsilon$, $r_\varepsilon$, $u^{BL}_\varepsilon$, $u^{CL}_\varepsilon$ are 
detailed in \cite[Thm.~{2.3.4}]{melenk02}.)
The desired approximation $\Pi u$ will be 
constructed of the form 
\begin{equation}
\Pi u = \Pi^{\infty}_{p,{\mathcal T}} w_\varepsilon + \Pi^{\infty}_{p,{\mathcal T}} r_\varepsilon + 
\Pi^{BL} \chi^{BL} u^{BL}_\varepsilon + 
\Pi^{CL} \chi^{CL} u^{CL}_\varepsilon.  
\end{equation}
We analyze these $4$ terms in turn.

\emph{Treatment of $w_\varepsilon$:} 
The function $w_\varepsilon$ is analytic with $\|\nabla^n w_\varepsilon\|_{L^\infty(\Omega)} \leq C \gamma^n n!
\quad \forall n \in \BbbN_0$ with constants $C$, $\gamma>0$ independent of $\varepsilon$. This implies 
for the shape-regular elements $K \in {\mathcal T}^{large} \cup {\mathcal T}^{CL}$ with $h_K$ denoting
the element diameter 
\begin{equation}
\label{eq:prop:thm.3.4.8-100}
h_K^{-1} \|w_\varepsilon - \Pi^{\infty}_{p,{\mathcal T}} w_\varepsilon\|_{L^\infty(K)} + 
\|\nabla( w_\varepsilon - \Pi^{\infty}_{p,{\mathcal T}} w_\varepsilon) \|_{L^\infty(K)}
\leq C e^{-bp}. 
\end{equation}
For the anisotropic elements $K \in {\mathcal T}^{aniso}$, we get 
\begin{equation}
\label{eq:prop:thm.3.4.8-200}
\|w_\varepsilon - \Pi^{\infty}_{p,{\mathcal T}} w_\varepsilon\|_{L^\infty(K)} + 
\kappa \|\nabla( w_\varepsilon - \Pi^{\infty}_{p,{\mathcal T}} w_\varepsilon) \|_{L^\infty(K)}
\leq C e^{-bp}. 
\end{equation}
Since $\kappa = \lambda p \varepsilon$, the estimates 
(\ref{eq:prop:thm.3.4.8-100}), (\ref{eq:prop:thm.3.4.8-200}) provide the desired estimates for the 
contribution of $w_\varepsilon$. 

\emph{Treatment of $r_\varepsilon$:} \cite[Thm. 2.3.4]{melenk02} gives 
$\|r_{\varepsilon}\|_{L^{\infty}(\Omega)}+\|r_{\varepsilon}\|_{H^{1}(\Omega)}\lesssim  e^{-b/\varepsilon}$. 
\eqref{eq:Pi_infty-2}, \eqref{eq:Pi_infty-3}, the fact that elements have aspect ratio at most 
${\mathcal O}(1/(\lambda p \varepsilon))$ and the assumption $\lambda p \varepsilon \lesssim 1$ imply  
(cf. also the arguments leading to \cite[(3.4.25)--(3.4.27)]{melenk02}) 
\begin{equation}
\|r_\varepsilon - \Pi^\infty_{p,{\mathcal T}} r_\varepsilon\|_{L^\infty(\Omega)} + 
\sqrt{\lambda p \varepsilon} \|\nabla (r_\varepsilon - \Pi^\infty_{p,{\mathcal T}} r_\varepsilon)\|_{L^2(\Omega)} 
\leq C p^2 C_p e^{-b/\varepsilon} 
\leq C p^2 C_p e^{-b \lambda p}. 
\end{equation}
Since $\varepsilon \lesssim 1/(\lambda p)$ this last inequality implies additionally 
\begin{equation}
\|\nabla (r_\varepsilon - \Pi^\infty_{p,{\mathcal T}} r_\varepsilon)\|_{L^2(\Omega)} 
\lesssim  \frac{1}{(\lambda p \varepsilon)^{1/2}} p^2 C_p e^{-b/\varepsilon} 
\lesssim  \frac{1}{\lambda^{1/2}} p^{3/2} C_p e^{-b/\varepsilon} 
\lesssim  \frac{1}{\lambda^{1/2}} p^{3/2} C_p e^{-b\lambda p} .
\end{equation}
These estimates provide the desired estimates for the contribution of $r_\varepsilon$. 

\emph{Treatment of $\chi^{BL} u^{BL}_\varepsilon$:} 
From \cite[(3.4.28)]{melenk02} we get for the boundary layer part $\chi^{BL} u^{BL}_\varepsilon$ 
\begin{equation}
\label{eq:prop:thm.3.4.8-1000}
\|\chi^{BL} u^{BL}_\varepsilon  - \Pi^\infty_{p,{\mathcal T}} \chi^{BL} u^{BL}_\varepsilon\|_{L^\infty(\Omega)}  
+ 
\lambda p \varepsilon 
\|\nabla (\chi^{BL} u^{BL}_\varepsilon  - 
             \Pi^\infty_{p,{\mathcal T}} \chi^{BL} u^{BL}_\varepsilon)\|_{L^\infty(\Omega)} 
\lesssim C_p p^2 e^{-b \lambda p}. 
\end{equation}
The approximation $\Pi^\infty_{p,{\mathcal T}} \chi^{BL} u^{BL}_\varepsilon$ needs to be corrected 
in the spirit of \cite{schwab-suri96} in order to control $\varepsilon^{1/2} \|\nabla(u - \Pi u)\|_{L^2(\Omega)}$.  
Specifically, we approximate $\chi^{BL} u^{BL}_\varepsilon$ by 
\begin{equation}
\Pi^{BL} \chi^{BL} u^{BL}_\varepsilon:= 
\begin{cases}
0  & x \in \Omega_0 \\
\Pi^\infty_{p,{\mathcal T}} \chi^{BL} u^{BL}_\varepsilon - {\mathcal L}_0 \left(\Pi^\infty_{p,{\mathcal T}} \chi^{BL} u^{BL}_\varepsilon\right) & 
x \in \Omega\setminus \Omega_0, 
\end{cases}
\end{equation}
where ${\mathcal L}_0$ is the lifting operator of Lemma~\ref{lemma:L0}. Note that 
$(\Pi^{BL} \chi^{BL} u^{BL}_\varepsilon) |_{\partial\Omega}  = 
\Pi^\infty_{p,{\mathcal T}} \chi^{BL} u^{BL}_\varepsilon$. 
For the analysis of the approximation properties of $\Pi^{BL} \chi^{BL} u^{BL}_\varepsilon$, we 
introduce the shorthand notation 
$$
\widetilde u^{BL}:= \chi^{BL} u^{BL}_\varepsilon 
\qquad \mbox{ and } \qquad 
\widetilde u^{BL}_p:= \Pi^{\infty}_{p,{\mathcal T}} (\chi^{BL} u^{BL}_\varepsilon). 
$$
We use that $\operatorname*{dist}(\Omega_0,\partial\Omega) \ge \mu \lambda p \varepsilon$ 
for some $\mu > 0$ (cf. (\ref{eq:Omega0-away-from-partialOmega})). 
The decay properties of the boundary layer (cf. \cite[Thm.~{2.3.4}]{melenk02}) then read 
\begin{equation}
\label{eq:prop:thm.3.4.8-500}
\varepsilon^{-1/2} \|\widetilde u^{BL}\|_{L^2(\Omega_0)} + 
\varepsilon^{1/2} \|\nabla \widetilde u^{BL}\|_{L^2(\Omega_0)} +
\|\widetilde u^{BL} \|_{L^\infty(\Omega_0)} + \varepsilon \|\nabla \widetilde u^{BL}\|_{L^\infty(\Omega_0)} \lesssim e^{-b \lambda p}. 
\end{equation}
These estimates and the fact 
$\|\cdot\|_{L^2(B_{\lambda p \varepsilon}(\overline{x})\cap \Omega_0)}\lesssim 
\lambda p \varepsilon \|\cdot\|_{L^{\infty}(B_{\lambda p \varepsilon}(\overline{x})\cap\Omega_0)}$ 
produce the correct estimates for the  approximation of $\chi^{BL} u^{BL}_\varepsilon$
on $\Omega_0$. 

In order to analyze the error on $\Omega\setminus\Omega_0$ we need to 
control $\Pi^{\infty}_{p,{\mathcal T}} (\chi^{BL} u^{BL}_\varepsilon)$ on $\partial\Omega_0$. To that end, 
we note that the stability properties of $\Pi^{\infty}_{p,{\mathcal T}}$ given in (\ref{eq:Pi_infty-2})
and the fact that the elements 
in $\Omega_0$ are shape-regular and of size ${\mathcal O}(1)$ imply
\begin{align}
\label{eq:prop:thm.3.4.8-600}
\|\widetilde u^{BL}_p\|_{L^\infty(\Omega_0)} & \lesssim C_p \|\widetilde u^{BL}\|_{L^\infty(\Omega_0)} 
\lesssim C_p e^{-b \lambda p}.
\end{align}
By construction, we have on $\Omega\setminus\Omega_0$ 
$$
\widetilde u^{BL} -\Pi^{BL} (\chi^{BL} u^{BL}_\varepsilon) = 
\widetilde u^{BL} - \widetilde u^{BL}_p + {\mathcal L}_0 \widetilde u^{BL}_p. 
$$
Since $\operatorname*{meas}(\Omega\setminus\Omega_0) = \mathcal{O}(\lambda p \varepsilon)$
and $\|{\cdot}\|_{L^2(\Omega\backslash\Omega_0)} \lesssim 
\operatorname*{meas}(\Omega\setminus\Omega_0)^{1/2}\|{\cdot}\|_{L^{\infty}(\Omega\backslash\Omega_0)}$, we get from 
(\ref{eq:prop:thm.3.4.8-1000}) the following estimates for the first term $\widetilde u^{BL} - \widetilde u^{BL}_p$:  
\begin{align}
(\lambda p \varepsilon)^{-1/2} \|\widetilde u^{BL} - \widetilde u^{BL}_p\|_{L^2(\Omega\setminus\Omega_0)} + 
(\lambda p \varepsilon)^{1/2} \|\nabla (\widetilde u^{BL} - \widetilde u^{BL}_p) \|_{L^2(\Omega\setminus \Omega_0)} 
\lesssim C_p p^2 e^{-b \lambda p }. 
\end{align}
For the term ${\mathcal L}_0 \widetilde u^{BL}_p$, we use the estimates of Lemma~\ref{lemma:L0}
and (\ref{eq:prop:thm.3.4.8-600}) to arrive at 
\begin{align}
\varepsilon^{1/2} \|\nabla {\mathcal L}_0 \widetilde u^{BL}_p\|_{L^2(\Omega\setminus\Omega_0)} +  
\varepsilon^{-1/2} \|{\mathcal L}_0 \widetilde u^{BL}_p\|_{L^2(\Omega\setminus\Omega_0)}   + 
\|{\mathcal L}_0 \widetilde u^{BL}_p\|_{L^\infty(\Omega)} &\leq 
\|{\mathcal L}_0\| \|\widetilde u^{BL}_p\|_{L^{\infty}(\partial\Omega_0)}
\nonumber\\
&\leq C p^{3/2} \lambda^{-1/2} C_p e^{-b \lambda p }. 
\end{align}
As in the proof of Lemma~\ref{lemma:L0} (cf. (\ref{eq:gradient})), we obtain, since 
$\operatorname*{meas}(B_{\lambda p \varepsilon}(\overline{x})) = \mathcal{O}((\lambda p\varepsilon)^2)$, 
\begin{equation*}
\|\nabla {\mathcal L}_0 \widetilde u^{BL}_p\|_{L^2(B_{\lambda p \varepsilon}(\overline{x})\cap(\Omega\setminus\Omega_0))} 
\lesssim \lambda \varepsilon p \|\nabla {\mathcal L}_0 \widetilde u^{BL}_p\|_{L^{\infty}(B_{\lambda p \varepsilon}(\overline{x})\cap(\Omega\setminus\Omega_0))} 
\lesssim p^2\|\widetilde u^{BL}_p\|_{L^{\infty}(\partial\Omega_0)} \leq C p^2 C_p e^{-b\lambda p},
\end{equation*}
which allows us to conclude that the approximation $\Pi^{BL} (\chi^{BL} u^{BL}_\varepsilon)$ has the desired
properties. 

\emph{Treatment of $u^{CL}_\varepsilon$:} 
Finally, for the corner layer contribution we use \cite[(3.4.29)---(3.4.33)]{melenk02}.  
Again, we abbreviate 
$$
\widetilde u^{CL}:= \chi^{CL} u^{CL}_\varepsilon, 
\qquad \mbox{ and } \qquad 
\widetilde u^{CL}_p= \Pi^\infty_{p,{\mathcal T}} (\chi^{CL} u^{CL}_\varepsilon). 
$$
We infer directly from \cite[(3.4.33)]{melenk02}:
\begin{align}
\label{eq:prop:thm.3.4.8-750}
\|\widetilde u^{CL} - \widetilde u^{CL}_p\|_{L^\infty(\Omega_{CL})} + 
\|\nabla( \widetilde u^{CL} - \widetilde u^{CL}_p)\|_{L^2(\Omega_{CL})} 
&\lesssim C_p p^3 e^{-b L} + e^{-bp}. 
\end{align}
As in our treatment of the boundary layer part, we need to modify the approximation 
$\Pi^{\infty}_{p,{\mathcal T}} \chi^{CL} u^{CL}_\varepsilon$. We set 
\begin{equation}
\Pi^{CL} \chi^{CL} u^{CL}_\varepsilon:= 
\begin{cases}
0 
& x \in \Omega\setminus\Omega_{CL} \\
\Pi^\infty_{p,{\mathcal T}} \chi^{CL} u^{CL}_\varepsilon -  {\mathcal L}_{CL} \left(\Pi^\infty_{p,{\mathcal T}} \chi^{CL} u^{CL}_\varepsilon\right) & 
x \in \Omega_{CL}, 
\end{cases}
\end{equation}
where ${\mathcal L}_{CL}$ is the lifting operator of Lemma~\ref{lemma:widetilde_L0}. Note that 
$(\Pi^{CL} \chi^{CL} u^{CL}_\varepsilon) |_{\partial\Omega}  = 
\Pi^\infty_{p,{\mathcal T}} \chi^{CL} u^{CL}_\varepsilon$. 

The decay properties of the corner layer (cf. \cite[Thm.~{2.3.4}]{melenk02}) then read 
\begin{equation}
\label{eq:prop:thm.3.4.8-1500}
\varepsilon^{-1} \|\widetilde u^{CL}\|_{L^2(\Omega\setminus\Omega_{CL})} + 
\|\nabla \widetilde u^{CL}\|_{L^2(\Omega\setminus\Omega_{CL})} +
\|\widetilde u^{CL} \|_{L^\infty(\Omega\setminus\Omega_{CL})} 
+ \varepsilon \|\nabla \widetilde u^{CL}\|_{L^\infty(\Omega\setminus\Omega_{CL})} 
\lesssim e^{-b \lambda p}. 
\end{equation}
These estimates imply that our approximation of the corner layer contribution 
has the desired properties on $\Omega\setminus\Omega_{CL}$. 

The stability properties of $\Pi^{\infty}_{p,{\mathcal T}}$ yield
\begin{align}
\label{eq:prop:thm.3.4.8-1600}
\|\widetilde u^{CL}_p\|_{L^\infty(\Omega\setminus\Omega_{CL})} & \lesssim 
C_p \|\widetilde u^{CL}\|_{L^\infty(\Omega\setminus\Omega_{CL})} 
\lesssim C_p e^{-b \lambda p}. 
\end{align}
By construction we have on $\Omega_{CL}$  
$$
\chi^{CL} u^{CL}_\varepsilon -\Pi^{CL} (\chi^{CL} u^{CL}_\varepsilon) = 
\widetilde u^{CL} - \widetilde u^{CL}_p +   {\mathcal L}_{CL} \widetilde u^{CL}_p. 
$$
The estimates of (\ref{eq:prop:thm.3.4.8-750}) give the desired bounds for 
the contribution $\widetilde u^{CL} - \widetilde u^{CL}_p$. For the 
correction ${\mathcal L}_{CL} \widetilde u^{CL}_p$ we use Lemma~\ref{lemma:widetilde_L0} 
and the bound (\ref{eq:prop:thm.3.4.8-1600}) to get 
\begin{align*}
\|\nabla {\mathcal L}_{CL}  \widetilde u^{CL}_p \|_{L^2(\Omega_{CL})} + 
\varepsilon^{-1} \|{\mathcal L}_{CL}  \widetilde u^{CL}_p \|_{L^2(\Omega_{CL})} +
\|{\mathcal L}_{CL}  \widetilde u^{CL}_p \|_{L^\infty(\Omega_{CL})}  &\lesssim \|{\mathcal L}_{CL}\|
\|\widetilde u^{CL}_p \|_{L^\infty(\partial\Omega_{CL})} \\
&\lesssim C_p p^2 e^{-b \lambda p}.
\qedhere 
\end{align*}
\end{proof}

\begin{theorem}
\label{thm:balanced-norm}
Assume the hypotheses of Proposition~\ref{prop:thm.3.4.8} and
let $\lambda_0$, which depends solely on 
$\Omega$, $A$, $c$, and $f$, be given 
by Proposition~\ref{prop:thm.3.4.8}. Then for each $\lambda \in (0,\lambda_0]$ there exist
$C$, $b > 0$ independent of $p$ and $\varepsilon$ such that for the  solution $u$ of (\ref{eq:model-problem}) 
and its Galerkin approximation $u_N \in S^{p,1}_0({\mathcal T}(\min\{\lambda p \varepsilon,1/2\},{\mathbf L}))$ 
there holds 
$$
\|u - u_N\|_{\sqrt{\varepsilon}} \leq C \left( e^{-b p}+ \sqrt{\varepsilon} p^{6} e^{-bL}\right). 
$$
\end{theorem}
\begin{proof}
With $\Pi u$ of Proposition~\ref{prop:thm.3.4.8}, the $L^{\infty}$-estimates of Proposition~\ref{prop:thm.3.4.8}
applied on $\Omega\backslash \Omega_{CL}$ and on $\Omega_{CL}$ together with
$\operatorname*{meas}(\Omega_{CL})  = {\mathcal O}((p \varepsilon)^2)$, we have
\begin{equation}
\label{eq:robust-a-priori-in-energy-norm}
\|u - u_N\|_{L^2(\Omega)} \leq 
\|u - u_N\|_\varepsilon \leq \inf_{v \in V_N} \|u - v\|_\varepsilon \leq \|u - \Pi u\|_\varepsilon \leq 
C p^4 \left( e^{-b p}+ \sqrt{\varepsilon} p^2 e^{-bL}\right). 
\end{equation}
We are therefore left with estimating $|u - u_N|_{\sqrt{\varepsilon}}$. To that end, 
we apply Corollary~\ref{cor:balanced-norm-key-estimate} with 
$v = \Pi u$ of Proposition~\ref{prop:thm.3.4.8}, which yields, since 
$\operatorname*{meas}(\Omega\setminus(\Omega_0\cup \Omega_{CL}))  = {\mathcal O}(p \varepsilon)$ 
and 
$\operatorname*{meas}(\Omega_{CL})  = {\mathcal O}((p \varepsilon)^2)$ \begin{align*} 
|u - u_N|_{\sqrt{\varepsilon}} &\lesssim 
\|{\mathcal L}_0\| \|\Pi^{L^2}_{\Omega_0} u - \Pi u\|_{L^\infty(\partial\Omega_0)} + 
\sqrt{\varepsilon} \|\nabla ( \Pi^{L^2}_{\Omega_0} u - \Pi u) \|_{L^2(\Omega_0)} 
+ p^{4.5} e^{-b p}+ \sqrt{\varepsilon}p^6 e^{-bL}. 
\end{align*} 
We exploit that $\Omega_0$ consists of a fixed number of shape-regular elements. Hence, a polynomial inverse
estimate on the reference element (cf.~\cite[(4.6.5)]{schwab98}) gives, since $\Omega_0 \cap \Omega_{CL} =\emptyset$,
$$
\|\nabla ( \Pi^{L^2}_{\Omega_0} u - \Pi u) \|_{L^2(\Omega_0)} 
\lesssim p^2 
\|\Pi^{L^2}_{\Omega_0} u - \Pi u \|_{L^2(\Omega_0)} 
\lesssim p^2 \|u - \Pi u\|_{L^2(\Omega_0)} \lesssim p^6 e^{-b p}. 
$$
The term $\|\Pi^{L^2}_{\Omega_0} u - \Pi u\|_{L^\infty(\partial\Omega_0)}$ is estimated 
again with polynomial inverse
estimates (cf. \cite[(4.6.1)]{schwab98})
\begin{align*}
\|\Pi^{L^2}_{\Omega_0} u - \Pi u\|_{L^\infty(\partial\Omega_0)}
\leq 
\|\Pi^{L^2}_{\Omega_0} u - \Pi u\|_{L^\infty(\Omega_0)}
\lesssim p^2 
\|\Pi^{L^2}_{\Omega_0} u - \Pi u\|_{L^2(\Omega_0)}
&\lesssim p^2 \|u - \Pi u\|_{L^2(\Omega_0)} \\
&\lesssim p^6  e^{-b p}. 
\end{align*}
Finally, Lemma~\ref{lemma:L0} yields $\|{\mathcal L}_0\| \leq C p^{3/2}$ for fixed $\lambda$. 
\end{proof}
\section{$L^\infty$-estimates}
$L^\infty$-estimates for the Galerkin error $u - u_N$ are obtained in $3$ steps: using the fact that the number of 
elements in $\Omega_0$ and in $\Omega_{aniso}$ is fixed, we estimate first $\|u - u_N\|_{L^\infty(\Omega_0)}$ and 
then $\|u - u_N\|_{L^\infty(\Omega_{aniso})}$. In a final step, we estimate $\|u - u_N\|_{L^\infty(\Omega_{CL})}$. 
For this last estimate, we need to make an assumption on the vector ${\mathbf L}$, namely, that significant 
geometric refinement is only possible at the boundary $\partial\Omega$: 
\begin{assumption} 
\label{assumption:geometric-refinement}
There is $L_\infty \ge 0$ such that for each $K^\M$ with $V:= F_{K^\M}(0,0) \in {\mathcal V}_{CL}$ the following 
dichotomy holds: either ($V \in \Omega$ and $L_{K^\M} \leq L_\infty$) or $V \in \partial\Omega$.  
\end{assumption}
\begin{lemma}
\label{lemma:use-of-extension-Linfty}
Let $V_N \subset H^1_0(\Omega)$ be a closed subspace. Let $\widetilde \Omega_0 \subset \Omega$ be open. 
Let $u \in H^1_0(\Omega)$ and $u_N \in V_N$ satisfy the Galerkin orthogonality (\ref{eq:Galerkin-orthogonality}). 
Let $Iu \in V_N$ satisfy $Iu|_{\widetilde\Omega_0} = u_N|_{\widetilde\Omega_0}$. 
Then, for an implied constant depending solely on $\|A\|_{L^\infty(\Omega)}$, $\|c\|_{L^\infty(\Omega)}$, 
$\alpha_0$, $c_0$ there holds 
\begin{equation}
|u_N - Iu|_{H^1(\Omega\setminus\widetilde\Omega_0)} 
\lesssim 
|u - Iu|_{H^1(\Omega\setminus\widetilde\Omega_0)}  + \varepsilon^{-1} 
\|u - Iu\|_{L^2(\Omega\setminus\widetilde\Omega_0)}. 
\end{equation}
\end{lemma}
\begin{proof}
Proceeds as in the proof of Lemma~\ref{lemma:use-L2-projection}.
\end{proof}
\begin{lemma}
\label{lemma:Linfty-vs-H1}
Let Assumption~\ref{assumption:geometric-refinement} be valid. 
Let $V_N:= S^{p,1}_0({\mathcal T}(\kappa,{\mathbf L}))$. 
Let $u_N$, $Iu \in V_N$ satisfy $(u_N - Iu)|_{\partial\Omega_{CL}} = 0$. Then
\begin{equation}
\label{eq:Linfty-vs-H1-10}
\|u_N - I u\|_{L^\infty(\Omega_{CL})} \leq C p 
\|\nabla (u_N-  Iu)\|_{L^2(\Omega_{CL})}.  
\end{equation}
The constant $C > 0$ depends only on the shape-regularity properties of ${\mathcal T}^\M$. 
\end{lemma}
\begin{proof}  
We note 
$$
\Omega_{CL} = \left(\bigcup_{V \in {\mathcal V}_{CL}} 
\bigcup \left\{F_{K^\M}([0,\kappa]^2)\,|\, \mbox{$K^\M \in {\mathcal T}^\M$ such that $F_{K^\M}(0,0) = V$}\right\}
\right)^\circ. 
$$
Consider a fixed $V \in {\mathcal V}_{CL}$. The cases $V \in \Omega$ or $V \in \partial\Omega$ may occur. 

\emph{The case $V \in \Omega$:} Assumption~\ref{assumption:geometric-refinement} implies that 
any macro-element $K^\M$ with $V = F_{K^\M}(0,0)$ is of tensor product, mixed, or geometric refinement type with 
$L_{K^\M} \leq L_\infty$. 
Denote by $\widehat u_N$ and $\widehat{Iu}$ the pull-backs of $u_N$ and $Iu$ to $\widehat S$ under
the patch map $F_{K^\M}$. 
%
We note that 
$\widehat e_1 = (0,\kappa) \times \{\kappa\}$ and 
$\widehat e_2 = \{\kappa\} \times (0,\kappa)$ form part of $F_{K^\M}^{-1} (\partial\Omega_{CL})$.
We iteratively apply the elementwise inverse estimates of Lemma~\ref{lemma:inverse-estimate} on the
pull-backs of the subelements of $\Omega_{CL}$ starting with the smallest elements. The boundary contributions on the right-hand side 
of Lemma~\ref{lemma:inverse-estimate} of the interior edges  
(in the last step these are the edges $\widehat e_1$, $\widehat e_2$) are estimated by the $L^{\infty}$-norm on the neighboring 
element, which in turn can again be estimated with Lemma~\ref{lemma:inverse-estimate}. After 
at most $L_{\infty}$ steps, we get
\begin{align*}
\|u_N - Iu\|_{L^\infty(K^{\M} \cap \Omega_{CL})}  & = 
\|\widehat u_N - \widehat{Iu}\|_{L^\infty((0,\kappa)^2)} 
\leq C p \|\nabla(\widehat u_N - \widehat{Iu})\|_{L^2((0,\kappa)^2)} + 
\|\widehat u_N - \widehat{Iu}\|_{L^\infty(\widehat e_1 \cup \widehat e_2)}  \\
&\lesssim p \|\nabla( u_N - Iu)\|_{L^2(K^\M \cap \Omega_{CL})} + 
\|u_N - Iu\|_{L^\infty(\partial\Omega_{CL})}. 
\end{align*}
This is the desired estimate since $\|u_N - Iu\|_{L^\infty(\partial\Omega_{CL})} = 0$ by assumption.

{\em The case $V \in \partial\Omega$:} 
Again, any macro-element $K^\M$ with $V = F_{K^\M}(0,0)$ is of tensor product, mixed, or geometric refinement type.  
Define the relevant neighborhood of $V$ by 
$$
\Omega_V:= \left( \bigcup_{K^\M \colon F_{K^\M}(0,0) = V} F_{K^\M}([0,\kappa]^2) \right)^\circ \subset \Omega_{CL}. 
$$
Since $V \in \partial\Omega$ and $\partial\Omega$ is a Lipschitz domain and $u_N - Iu \in H^1_0(\Omega_V)$, 
Lemma~\ref{lemma:weighted-poincare} will be applicable. 

Fix a $K^\M$ with $F_{K^\M}(0,0) = V$. 
Denote by $\widehat u_N$ and $\widehat{Iu}$ the pull-backs of $u_N$ and $Iu$ to $\widehat S$ under
the patch map $F_{K^\M}$. 
We use polynomial inverse estimates
(cf. \cite[(4.6.3)]{schwab98}) and scaling arguments for each element $K \in \widehat{\mathcal T}^{CL}_{K^\M}$ 
of that patch to estimate with $h_K$ denoting the element size of $K \in \widehat{\mathcal T}^{CL}_{K^\M}$: 
$$
\|\widehat u_N - \widehat{Iu}\|_{L^\infty(K)} \lesssim 
\sqrt{\ln (p+1)} \left[ h_K^{-1} \|\widehat u_N - \widehat{Iu}\|_{L^2(K)} 
+ \|\nabla (\widehat u_N - \widehat{Iu})\|_{L^2(K)}\right]. 
$$
Denoting by $\widehat r$ the distance from the origin and recalling that $\widehat{\mathcal T}^{CL}_{K^\M}$ is a 
geometric mesh so that for $K \in \widehat{\mathcal T}^{CL}_{K^\M}$ with $(0,0) \not\in \overline{K} $ 
we have $h_K \sim \widehat r(x)$ for all $x \in K$, we can estimate 
\begin{align*}
\|\widehat u_N - \widehat{Iu}\|_{L^\infty(K)} & 
\lesssim \sqrt{\ln (p+1)} \left[ \left\|\frac{1}{\widehat r} (\widehat u_N - \widehat{Iu})\right\|_{L^2(K)} 
+ \|\nabla (\widehat u_N - \widehat{Iu})\|_{L^2(K)}\right] \\
&\lesssim \sqrt{\ln (p+1)} \left[ \left\|\frac{1}{\widehat r} (\widehat u_N - \widehat{Iu})\right\|_{L^2((0,\kappa)^2)} \
+ \|\nabla (\widehat u_N - \widehat{Iu})\|_{L^2((0,\kappa)^2)}\right]. 
\end{align*}
Denoting by $r_V$ the distance from $V$, we conclude 
\begin{align*}
\|u_N - {Iu}\|_{L^\infty(\Omega_V)}  
&\lesssim \sqrt{\ln (p+1)} \left[ \left\|\frac{1}{r_V} (u_N - {Iu})\right\|_{L^2(\Omega_V)} 
+ \|\nabla (u_N - {Iu})\|_{L^2(\Omega_V)}\right] \\ 
&\lesssim \sqrt{\ln (p+1)} \left\|\nabla (u_N - {Iu})\right\|_{L^2(\Omega_V)},  
\end{align*}
where the second inequality follows from Lemma~\ref{lemma:weighted-poincare}, our assumption that 
$V \in \partial\Omega$, and the observation that $u_N - Iu \in H^1_0(\Omega_V)$. 
\end{proof}
\begin{lemma}
\label{lemma:Linfty-Omega0-Omegaaniso}
Assume the hypotheses of Proposition~\ref{prop:thm.3.4.8} and let $\lambda_0$ be given 
by Proposition~\ref{prop:thm.3.4.8}. Then for each $\lambda \in (0,\lambda_0]$ there exist
$C$, $b > 0$ independent of $p$ and $\varepsilon$ such that the 
Galerkin error $u - u_N$ 
satisfies 
\begin{align*}
\|u - u_N\|_{L^\infty(\Omega_0)} &\leq  C \left(  e^{-bp} +\sqrt{\varepsilon} p^8 e^{-bL}\right)
& \mbox{ and } & &
\|u - u_N\|_{L^\infty(\Omega_{aniso})} &\leq C \left(e^{-bp}+\sqrt{\varepsilon} p^{8} e^{-bL}\right).
\end{align*}
\end{lemma}
\begin{proof}
Let $\Pi u$ be the approximation of Proposition~\ref{prop:thm.3.4.8}. Then, using that $\Omega_0$ is the union of 
large, shape-regular elements, we get  
\begin{align*}
\|u - u_N\|_{L^\infty(\Omega_0)} &\leq 
\|u - \Pi u\|_{L^\infty(\Omega_0)} + 
\|\Pi u - u_N\|_{L^\infty(\Omega_0)} 
\lesssim 
\|u - \Pi u\|_{L^\infty(\Omega_0)} + 
p^2 \|\Pi u - u_N\|_{L^2(\Omega_0)}\\
&\lesssim  
\|u - \Pi u\|_{L^\infty(\Omega_0)} + 
p^2 \|u - \Pi u\|_{L^2(\Omega_0)} + 
p^2 \|u - u_N\|_{L^2(\Omega_0)}
 \lesssim e^{-b p}+\sqrt{\varepsilon}p^8 e^{-bL}, 
\end{align*}
where the last step used Proposition~\ref{prop:thm.3.4.8} and employed (\ref{eq:robust-a-priori-in-energy-norm}). 
Next, we exploit that each element in ${\mathcal T}^{aniso}$ shares a ``long'' edge with either $\partial\Omega$
or with $\partial\Omega_0$. This implies with polynomial inverse estimates 
(cf.~Lemma~\ref{lemma:inverse-estimate} applied with $h_y = \lambda p\varepsilon$, $h_x = {\mathcal O}(1)$), 
Proposition~\ref{prop:thm.3.4.8}, and $\Omega_{aniso} \cap \Omega_{CL} = \emptyset$
\begin{align*}
\|u - u_N\|_{L^\infty(\Omega_{aniso})} &\leq 
\|u - \Pi u\|_{L^\infty(\Omega_{aniso})} + 
\|\Pi u - u_N\|_{L^\infty(\Omega_{aniso})}  \\
& \lesssim 
\|u - \Pi u\|_{L^\infty(\Omega_{aniso})} + 
(\lambda p \varepsilon)^{1/2} p \|\nabla(\Pi u - u_N)\|_{L^2(\Omega_{aniso})} + \|\Pi u - u_N\|_{L^\infty(\partial\Omega_0)} \\
&\lesssim \|u - u_N\|_{L^\infty(\Omega_0)} + p^{3/2} \|u - u_N\|_{\sqrt{\varepsilon}} 
+ e^{-bp} +\sqrt{\varepsilon}p^{7.5}e^{-bL} \\
& \stackrel{\text{Thm.~\ref{thm:balanced-norm}}}{\lesssim} e^{-bp} + \sqrt{\varepsilon}p^{8}e^{-bL}. 
\qedhere
\end{align*}
\end{proof}
\begin{theorem}
\label{thm:Linfty-estimate}
Assume the hypotheses of Proposition~\ref{prop:thm.3.4.8}
and let $\ell > 0$. 
Let Assumption~\ref{assumption:geometric-refinement} be valid. 
Assume that ${L}_{K^\M} \ge \ell p$ for those macro-elements $K^\M$ with the property 
that $F_{K^\M}(0,0)$ is a vertex of $\Omega$.
Let $\lambda_0$ be given 
by Proposition~\ref{prop:thm.3.4.8}. Let $u \in H^1_0(\Omega)$ solve (\ref{eq:model-problem}) and 
$u_N \in S^{p,1}_0({\mathcal T}(\min\{\lambda p \varepsilon,1/2\},{\mathbf L}))$ be its Galerkin approximation. 
Then for each $\lambda \in (0,\lambda_0]$ there exist
$C$, $b > 0$ independent of $p$ and $\varepsilon$ such that the finite element error $u - u_N$ satisfies 
$$
\|u - u_N\|_{L^\infty(\Omega)} \leq C e^{-bp}. 
$$
\end{theorem}
\begin{proof}
In view of Lemma~\ref{lemma:Linfty-Omega0-Omegaaniso} and $L=\mathcal{O}(p)$ it suffices to estimate $\|u - u_N\|_{L^\infty(\Omega_{CL})}$. Define, with $\Pi u$ given
by Proposition~\ref{prop:thm.3.4.8}, the function $Iu \in V_N$ by 
$$
Iu:= 
\begin{cases}
u_N & x \in \Omega \setminus \Omega_{CL} \\
\Pi u - {\mathcal L}_{CL} (\Pi u-u_N) & x \in \Omega_{CL} 
\end{cases}
$$
and estimate 
\begin{align}
\label{eq:Linftytmp}
\|u - u_N\|_{L^\infty(\Omega_{CL})} &\leq 
\|u - \Pi u\|_{L^\infty(\Omega_{CL})} + \|u_N - Iu\|_{L^\infty(\Omega_{CL})} + 
\|{\mathcal L}_{CL}(\Pi u - u_N)\|_{L^\infty(\Omega_{CL})}.
\end{align} 
The term $\|u - \Pi u\|_{L^\infty(\Omega_{CL})}$ is estimated in the desired form in Proposition~\ref{prop:thm.3.4.8}. 
For the second term in (\ref{eq:Linftytmp}) we note 
\begin{align*}
& \|u_N - Iu\|_{L^\infty(\Omega_{CL})} 
 \stackrel{\text{L.~\ref{lemma:Linfty-vs-H1}}}{\lesssim}
p \|\nabla( u_N - Iu)\|_{L^2(\Omega_{CL})} 
 \stackrel{\text{L.~\ref{lemma:use-of-extension-Linfty}}} {\lesssim}
 p \left[ \|\nabla(u - Iu)\|_{L^2(\Omega_{CL})} + \varepsilon^{-1} \|u - I u\|_{L^2(\Omega_{CL})}\right] \\
\quad & {\lesssim}
 p \left[ \|\nabla(u - \Pi u) \|_{L^2(\Omega_{CL})} + \|\nabla (\Pi u - Iu)\|_{L^2(\Omega_{CL})} + 
\varepsilon^{-1} \|u - \Pi u\|_{L^2(\Omega_{CL})} + 
\varepsilon^{-1} \|\Pi u - I u\|_{L^2(\Omega_{CL})}\right]. 
\end{align*}
Again, the terms involving $u - \Pi u$ can be estimated in the desired fashion using Proposition~\ref{prop:thm.3.4.8}. 
The remaining terms involving $\Pi u - Iu$ together with the third term of \eqref{eq:Linftytmp} 
are treated as follows: 
\begin{align*}
&\|\nabla(\Pi u - Iu)\|_{L^2(\Omega_{CL})} + \varepsilon^{-1} \|\Pi u - Iu\|_{L^2(\Omega_{CL})} +\|{\mathcal L}_{CL}(\Pi u - u_N)\|_{L^\infty(\Omega_{CL})}
\\&\qquad = 
\|\nabla {\mathcal L}_{CL} (u_N - \Pi u) \|_{L^2(\Omega_{CL})} + \varepsilon^{-1} \|{\mathcal L}_{CL} (u_N-  \Pi u)\|_{L^2(\Omega_{CL})}
+\|{\mathcal L}_{CL}(\Pi u - u_N)\|_{L^\infty(\Omega_{CL})}
\\ 
&\qquad\leq \|{\mathcal L}_{CL} \| \|u_N - \Pi u\|_{L^\infty(\partial\Omega_{CL})} \\
&\qquad\lesssim 
\|{\mathcal L}_{CL} \| \left[ \|u - u_N\|_{L^\infty(\partial\Omega_{CL})} + \|u - \Pi u\|_{L^\infty(\partial\Omega_{CL})}\right] \\
&\qquad\lesssim 
\|{\mathcal L}_{CL} \| \left[ \|u - u_N\|_{L^\infty(\Omega\setminus\Omega_{CL})} + \|u - \Pi u\|_{L^\infty(\partial\Omega_{CL})}\right] . 
\end{align*}
The second term can be controlled with the aid of Proposition~\ref{prop:thm.3.4.8}. 
Since $\overline{\Omega\setminus\Omega_{CL} } = \overline{\Omega_0} \cup \overline{\Omega_{aniso}}$, 
the first term can be controlled using Lemma~\ref{lemma:Linfty-Omega0-Omegaaniso}. 
\end{proof}
\section{Numerical example}
\label{sec:numerics}
We provide numerical examples that underline the robust exponential 
convergence of the $hp$-FEM solution in the balanced norm.
On the L-shaped domain 
$\Omega := (0,1)^2 \setminus ([1/2,1)\times [1/2,1))$  we study 
\begin{equation}
\label{eq:model-problemNumerics} 
-\varepsilon^2 \Delta u  + u = f \quad \mbox{ in $\Omega$}, 
\qquad u|_{\partial\Omega} = 0. 
\end{equation}

We use a spectral boundary layer mesh $\mathcal{T}(p\varepsilon,\mathbf{p+1})$ that is visualized 
in Fig.~\ref{fig:meshsolution} (left) and is designed in the spirit of the meshes described
in Section~\ref{sec:spectral-boundary-layer-mesh}. Although it consists of triangles only and is derived from 
two types of refinement patterns that are not covered by Definition~\ref{def:admissible-patterns}, the above analysis
could be extended to cover this type of mesh. 
The vector $\mathbf{p+1}$ stands for the constant vector with entries $p+1$ and 
reflects the fact that we employ $p+1$ steps of geometric refinement towards each of the $6$ vertices of the domain. 

\begin{figure}[ht]
\psfragscanon
\quad
\includegraphics[width=0.42\textwidth]{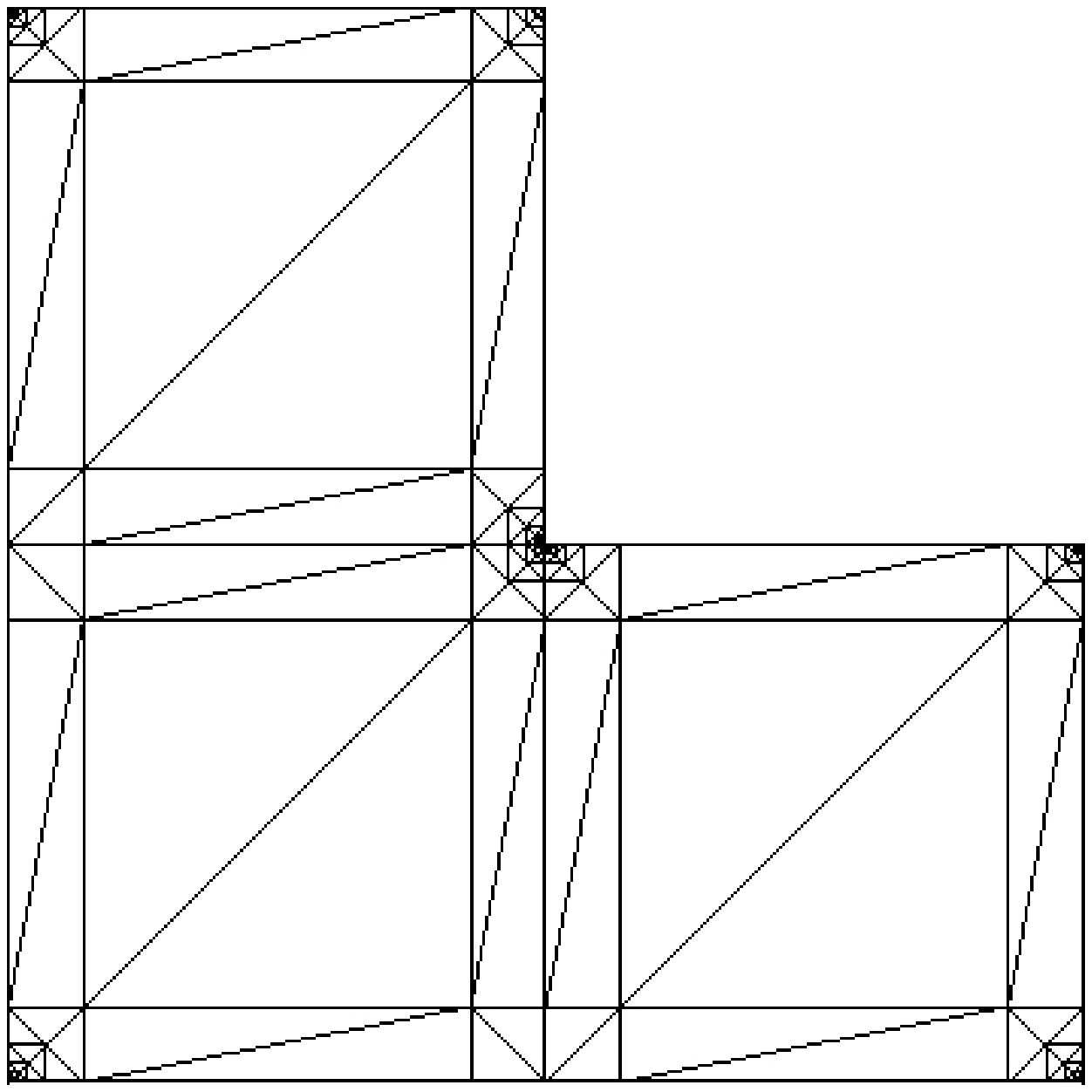}
\hfill
\includegraphics[width=0.42\textwidth]{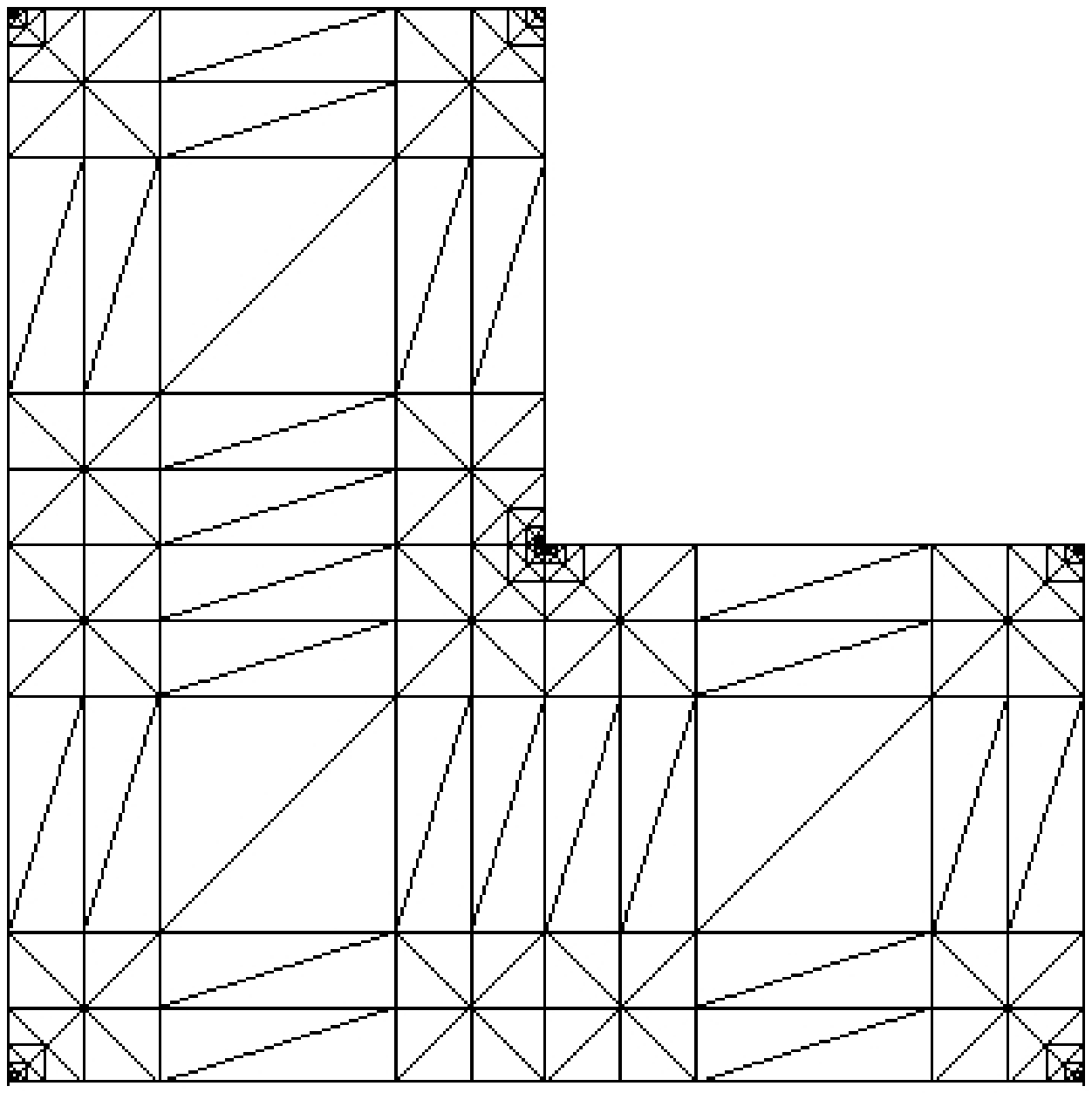}
\quad
\psfragscanoff
\caption{
\label{fig:meshsolution} Left: spectral boundary layer mesh. Right: refined mesh for computing reference solution.
}
\end{figure}

The finite element approximation $u_N \in S^{p,1}_0(\mathcal{T}(p\varepsilon,\mathbf{p+1}))$ is computed 
with the C++-software package NGSOLVE, \cite{schoeberl97,schoeberl-ngsolve,schoeberl14} 
for $p=1,\ldots,p_{max}:=7$. As the exact solution $u$ is unknown, we 
compute a reference solution $u_{ref}$ on a grid $\mathcal{T}_{fine}$, which is indicated in 
Fig.~\ref{fig:meshsolution} (right);  it is a refinement of $\mathcal{T}(p\varepsilon,\mathbf{p+1})$ obtained by adding a second layer of 
anisotropic elements around the boundary layer and doing two additional steps of geometric refinement to 
the corners. Additionally, the reference solution 
on this grid is computed with a polynomial degree of $2p_{max}$. 

\begin{example}
\label{example:polynomial-f}
\rm 
We select $f \equiv 1$ in (\ref{eq:model-problemNumerics}). 
Fig.~\ref{fig:errorbalanced} shows $\|u_{ref} - u_N\|_{L^2(\Omega)}$ and 
$|u_{ref} - u_N|_{\sqrt{\varepsilon}}$ versus the polynomial degree $p$ 
(see Tables~\ref{table:L2-polynomial}, \ref{table:balanced-polynomial} for tables with the results).
An exponential decay that is robust
in $\varepsilon$ is visible. The $L^2$-error even appears to scale with $\sqrt{\varepsilon}$. 
The balanced norm is defined in (\ref{eq:balanced-norm}) as the sum of both contributions and features therefore
also robust exponential convergence in $p$. 
\eremk
\end{example}

\begin{figure}[h]
\psfragscanon
\quad
\psfrag{L2-error}{$L^2(\Omega)$-error}
\psfrag{Polynomial degree}{\footnotesize polynomial degree $p$}
\psfrag{Error}{\footnotesize error}
\psfrag{error}{\footnotesize error}
\psfrag{e=1e-02}{\footnotesize $\varepsilon=10^{-2}$}
\psfrag{e=1e-03}{\footnotesize $\varepsilon=10^{-3}$}
\psfrag{e=1e-04}{\footnotesize $\varepsilon=10^{-4}$}
\psfrag{e=1e-05}{\footnotesize $\varepsilon=10^{-5}$}
\psfrag{e=1e-06}{\footnotesize $\varepsilon=10^{-6}$}
\psfrag{e=1e-07}{\footnotesize $\varepsilon=10^{-7}$}
\psfrag{e=1e-08}{\footnotesize $\varepsilon=10^{-8}$}
\includegraphics[width=0.49\textwidth]{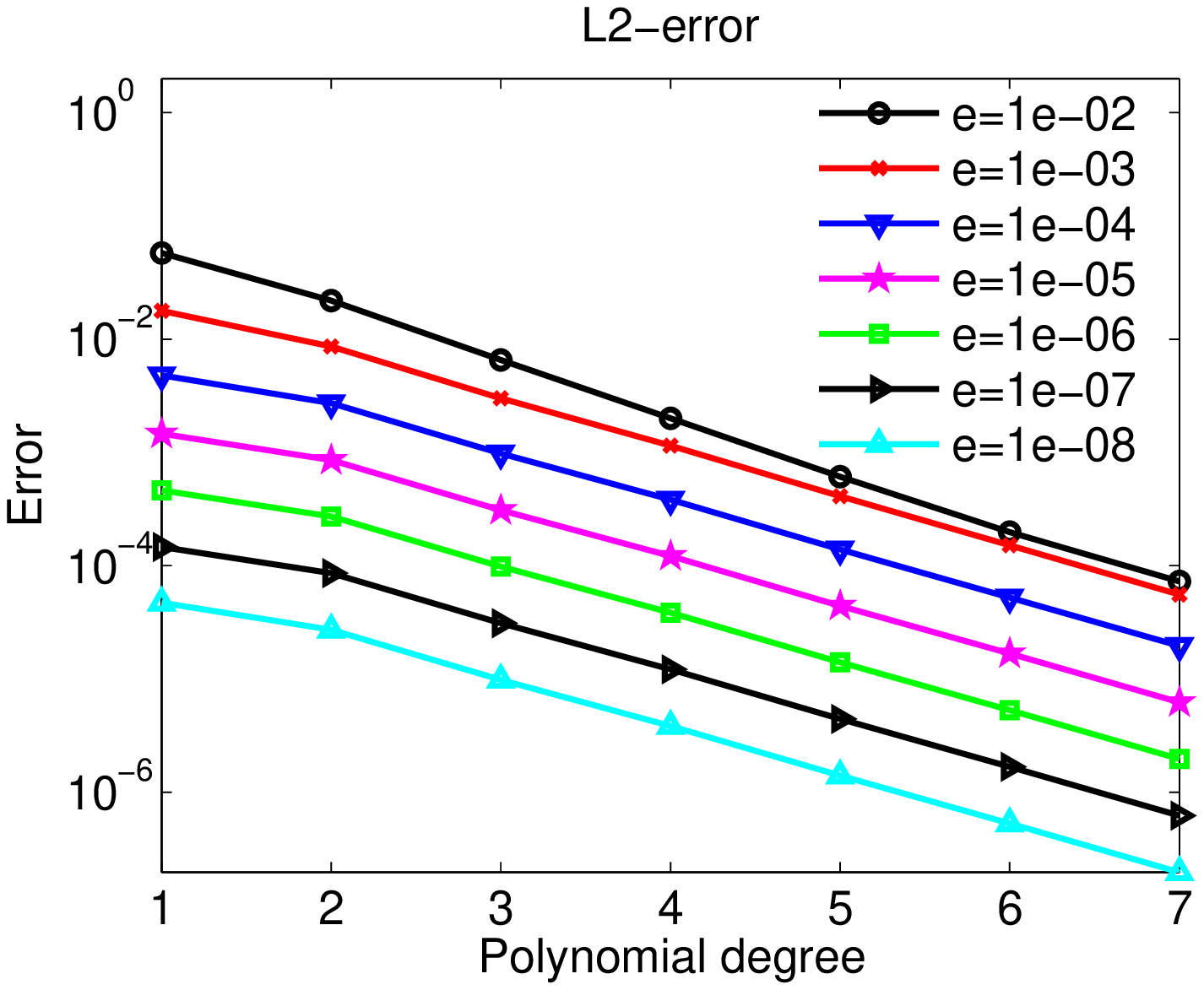}
\hfill
\psfrag{balH1Semi-Error}{$\sqrt{\varepsilon}\cdot |{\cdot}|_{H^1(\Omega)}$-error}
\includegraphics[width=0.49\textwidth]{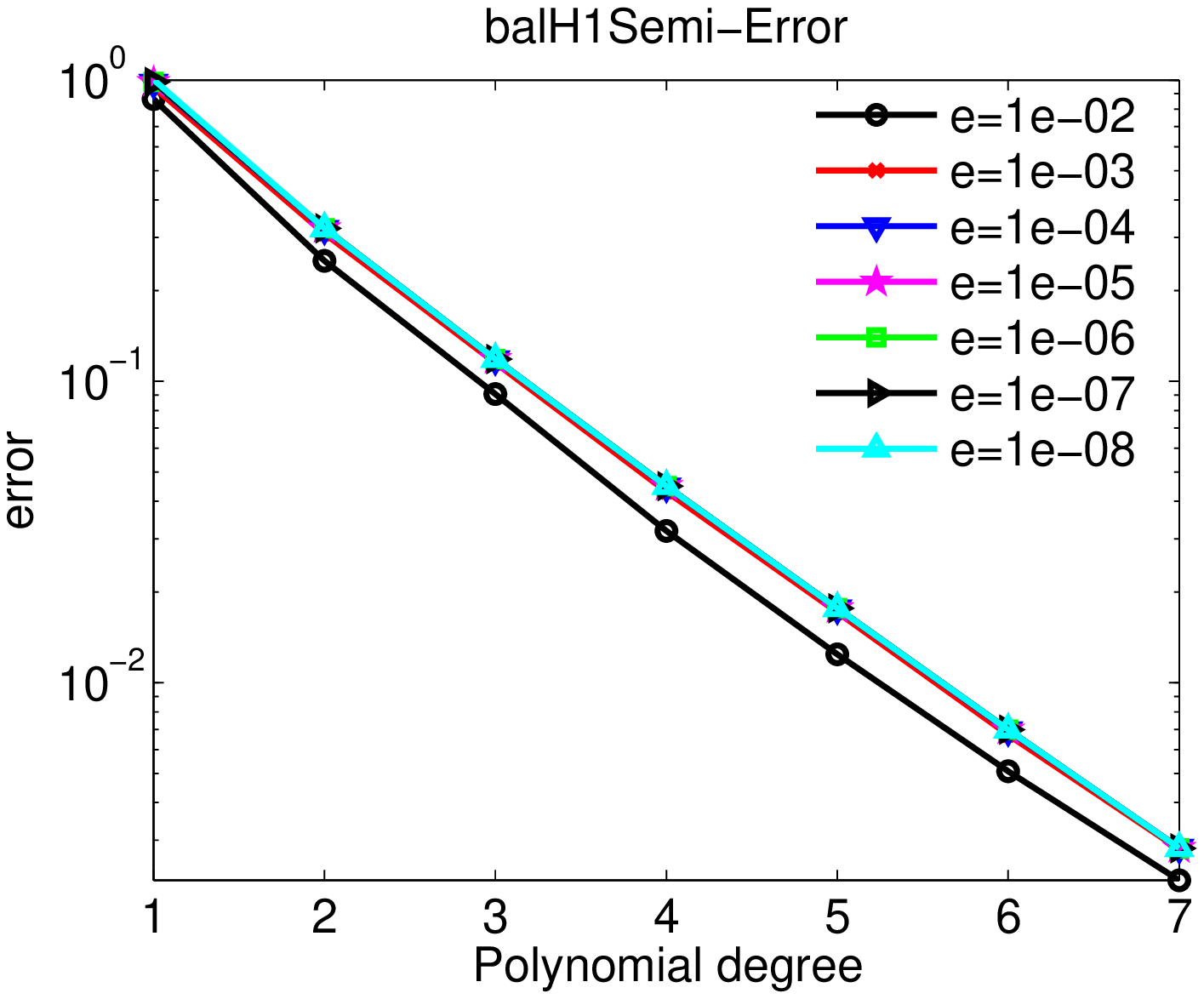}
\quad
\psfragscanoff
\caption{
\label{fig:errorbalanced} (cf.~Ex.~\ref{example:polynomial-f}) Left: $L^2$-error $\|u_N-u_{ref}\|_{L^2(\Omega)}$. Right:
balanced $H^1$-seminorm error $\sqrt{\varepsilon}\|\nabla(u_N-u_{ref})\|_{L^2(\Omega)}$.} 
\end{figure}

\begin{example}
\label{example:nearly-singular-f}
\rm 
We select $f(x,y) = \frac{1}{x^2+y^2+0.15}$  in (\ref{eq:model-problemNumerics}). 
We use the same mesh as in Example~\ref{example:polynomial-f}. 
Fig.~\ref{fig:errorbalanced1overr} shows again the errors in $L^2$ and the balanced $H^1$-seminorm
(see Tables~\ref{table:L2-nearly-singular}, \ref{table:balanced-nearly-singular} for tables with the results).
In contrast to Example~\ref{example:polynomial-f}, there is no significant dependence on $\varepsilon$ 
in the $L^2$-norm. A possible explanation is that the $L^2$-error can be bounded in the form 
$e^{-b_1 p} + \sqrt{\varepsilon} e^{-b_2 p}$, where the first term may be associated with $\Omega_0$ 
whereas the second term is linked to $\Omega\setminus\Omega_0$. The asymptotically dominant convergence 
depends on whether $b_1$ or $b_2$ is smaller. 
\eremk
\end{example}
\begin{figure}[h]
\psfragscanon
\quad
\psfrag{L2-error}{$L^2(\Omega)$-error}
\psfrag{Polynomial degree}{\footnotesize polynomial degree $p$}
\psfrag{Error}{\footnotesize error}
\psfrag{error}{\footnotesize error}
\psfrag{e=1e-02}{\footnotesize $\varepsilon=10^{-2}$}
\psfrag{e=1e-03}{\footnotesize $\varepsilon=10^{-3}$}
\psfrag{e=1e-04}{\footnotesize $\varepsilon=10^{-4}$}
\psfrag{e=1e-05}{\footnotesize $\varepsilon=10^{-5}$}
\psfrag{e=1e-06}{\footnotesize $\varepsilon=10^{-6}$}
\psfrag{e=1e-07}{\footnotesize $\varepsilon=10^{-7}$}
\psfrag{e=1e-08}{\footnotesize $\varepsilon=10^{-8}$}
\includegraphics[width=0.49\textwidth]{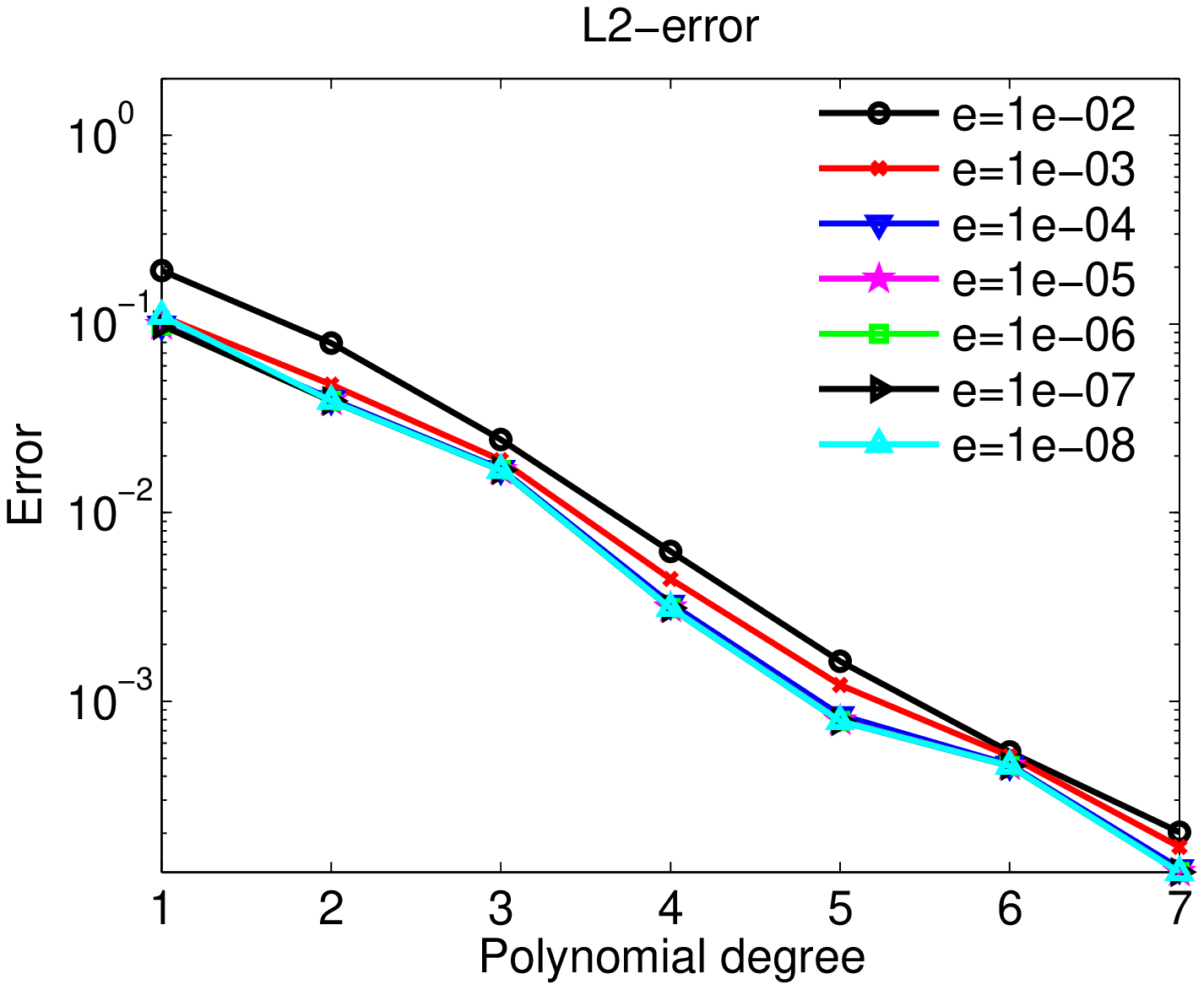}
\hfill
\psfrag{balH1Semi-Error}{$\sqrt{\varepsilon}\cdot |{\cdot}|_{H^1(\Omega)}$-error}
\includegraphics[width=0.49\textwidth]{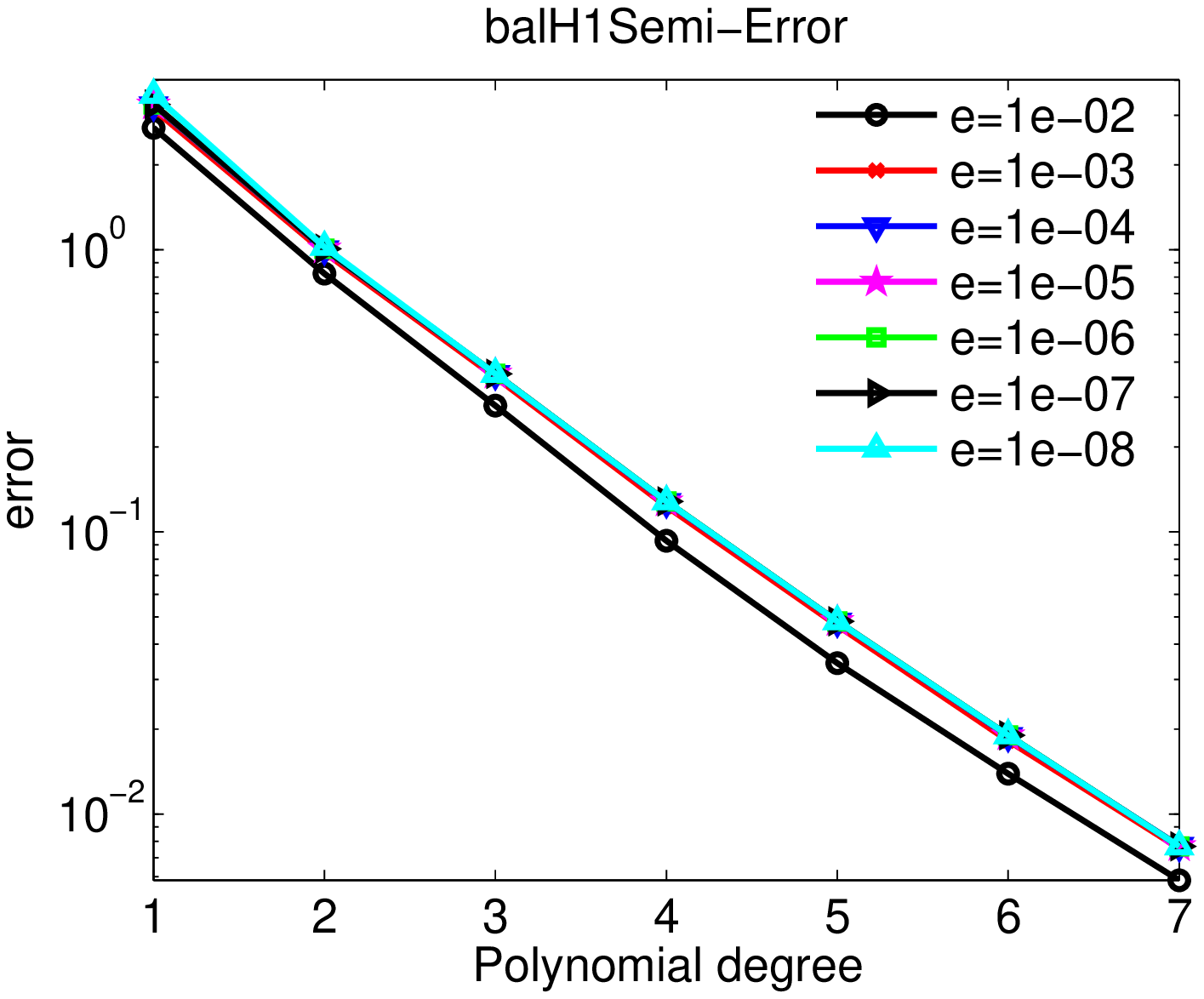}
\quad
\psfragscanoff
\caption{
\label{fig:errorbalanced1overr} (cf.~Ex.~\ref{example:nearly-singular-f}) 
Left: $L^2$-error $\|u_N-u_{ref}\|_{L^2(\Omega)}$. Right:  
balanced $H^1$-seminorm error $\sqrt{\varepsilon}\|\nabla(u_N-u_{ref})\|_{L^2(\Omega)}$.}  
\end{figure}

\section*{Appendix A}
\def\thesection{A}
\setcounter{Theorem}{0}
\begin{table}[h]
\begin{center}
{\footnotesize
\begin{tabular}{|c|c|c|c|c|c|c|c|}
\hline
 $p/\varepsilon$ & 1e-02 & 1e-03 & 1e-04 & 1e-05 & 1e-06 & 1e-07 & 1e-08\\ 
\hline 1 & 5.76e-02 & 1.78e-02 & 4.78e-03 & 1.47e-03 & 4.62e-04 & 1.46e-04 & 4.70e-05\\ 
\hline 2 & 2.19e-02 & 8.60e-03 & 2.71e-03 & 8.57e-04 & 2.71e-04 & 8.57e-05 & 2.71e-05\\ 
\hline 3 & 6.55e-03 & 3.00e-03 & 9.77e-04 & 3.10e-04 & 9.80e-05 & 3.10e-05 & 9.80e-06\\ 
\hline 4 & 2.00e-03 & 1.15e-03 & 3.82e-04 & 1.22e-04 & 3.85e-05 & 1.22e-05 & 3.85e-06\\ 
\hline 5 & 6.09e-03 & 4.11e-04 & 1.39e-04 & 4.42e-05 & 1.40e-05 & 4.43e-06 & 1.40e-06\\ 
\hline 6 & 1.97e-04 & 1.51e-04 & 5.24e-05 & 1.68e-05 & 5.31e-06 & 1.68e-06 & 5.31e-07\\ 
\hline 7 & 7.27e-05 & 5.54e-05 & 1.96e-05 & 6.22e-06 & 1.97e-06 & 6.22e-07 & 1.97e-07\\ 
\hline\end{tabular}
}
\caption{\label{table:L2-polynomial} $L^2(\Omega)$-errors for Example~\ref{example:polynomial-f}.}
\end{center}
\begin{center}
{\footnotesize
\begin{tabular}{|c|c|c|c|c|c|c|c|}
\hline
 $p/\varepsilon$ & 1e-02 & 1e-03 & 1e-04 & 1e-05 & 1e-06 & 1e-07 & 1e-08\\ 
\hline 1 & 8.64e-01 & 9.48e-01 & 9.79e-01 & 9.85e-01 & 9.85e-01 & 9.85e-01 & 1.00e+00\\ 
\hline 2 & 2.51e-01 & 3.10e-01 & 3.20e-01 & 3.22e-01 & 3.22e-01 & 3.22e-01 & 3.22e-01\\ 
\hline 3 & 9.07e-02 & 1.15e-01 & 1.18e-01 & 1.19e-01 & 1.19e-01 & 1.19e-01 & 1.19e-01\\ 
\hline 4 & 3.19e-02 & 4.31e-02 & 4.47e-02 & 4.49e-02 & 4.49e-02 & 4.49e-02 & 4.49e-02\\ 
\hline 5 & 1.24e-02 & 1.71e-02 & 1.76e-02 & 1.77e-02 & 1.77e-02 & 1.77e-02 & 1.77e-02\\ 
\hline 6 & 5.08e-03 & 6.71e-03 & 6.94e-03 & 6.98e-03 & 6.98e-03 & 6.98e-03 & 6.98e-03\\ 
\hline 7 & 2.21e-03 & 2.80e-03 & 2.83e-03 & 2.82e-03 & 2.82e-03 & 2.82e-03 & 2.82e-03\\ 
\hline\end{tabular}}
\caption{\label{table:balanced-polynomial} Balanced $H^1$-seminorm errors for Example~\ref{example:polynomial-f}.}
\end{center}
\end{table}

\begin{table}[h]
\begin{center}
{\footnotesize
\begin{tabular}{|c|c|c|c|c|c|c|c|}
\hline
 $p/\varepsilon$ & 1e-02 & 1e-03 & 1e-04 & 1e-05 & 1e-06 & 1e-07 & 1e-08\\ 
\hline 1 & 1.92e-01 & 1.10e-01 & 9.94e-02 & 9.86e-02 & 9.85e-02 & 9.85e-02 & 1.10e-01\\ 
\hline 2 & 7.92e-02 & 4.77e-02 & 4.01e-02 & 3.92e-02 & 3.91e-02 & 3.91e-02 & 3.91e-02\\ 
\hline 3 & 2.44e-02 & 1.90e-02 & 1.70e-02 & 1.68e-02 & 1.68e-02 & 1.68e-02 & 1.68e-02\\ 
\hline 4 & 6.23e-03 & 4.45e-03 & 3.29e-03 & 3.14e-03 & 3.12e-03 & 3.12e-03 & 3.12e-03\\ 
\hline 5 & 1.63e-03 & 1.22e-03 & 8.46e-04 & 7.93e-04 & 7.88e-04 & 7.87e-04 & 7.87e-04\\ 
\hline 6 & 5.39e-04 & 5.13e-04 & 4.63e-04 & 4.56e-04 & 4.56e-04 & 4.56e-04 & 4.56e-04\\ 
\hline 7 & 2.03e-04 & 1.69e-04 & 1.31e-04 & 1.25e-04 & 1.24e-04 & 1.24e-04 & 1.24e-04\\ 
\hline\end{tabular}}
\caption{\label{table:L2-nearly-singular} $L^2(\Omega)$-errors for Example~\ref{example:nearly-singular-f}.}
\end{center}
\begin{center}
{\footnotesize
\begin{tabular}{|c|c|c|c|c|c|c|c|}
\hline
 $p/\varepsilon$ & 1e-02 & 1e-03 & 1e-04 & 1e-05 & 1e-06 & 1e-07 & 1e-08\\ 
\hline 1 & 2.71e+00 & 3.15e+00 & 3.26e+00 & 3.27e+00 & 3.27e+00 & 3.27e+00 & 3.54e+00\\ 
\hline 2 & 8.23e-01 & 9.78e-01 & 1.00e+00 & 1.01e+00 & 1.01e+00 & 1.01e+00 & 1.02e+00\\ 
\hline 3 & 2.80e-01 & 3.54e-01 & 3.62e-01 & 3.64e-01 & 3.64e-01 & 3.64e-01 & 3.64e-01\\ 
\hline 4 & 9.90e-02 & 1.23e-01 & 1.28e-01 & 1.28e-01 & 1.28e-01 & 1.28e-01 & 1.28e-01\\ 
\hline 5 & 3.43e-02 & 4.66e-02 & 4.81e-02 & 4.82e-02 & 4.82e-02 & 4.82e-02 & 4.82e-02\\ 
\hline 6 & 1.39e-02 & 1.82e-02 & 1.89e-02 & 1.90e-02 & 1.90e-02 & 1.90e-02 & 1.90e-02\\ 
\hline 7 & 5.83e-03 & 7.56e-03 & 7.71e-03 & 7.69e-03 & 7.68e-03 & 7.68e-03 & 7.68e-03\\ 
\hline\end{tabular}}
\caption{\label{table:balanced-nearly-singular} 
Balanced $H^1$-seminorm errors for Example~\ref{example:nearly-singular-f}.}
\end{center}
\end{table}

\begin{lemma}
\label{lemma:lifting} 
Let $\widehat K$ be the reference triangle or the reference square. Then there exists $C > 0$ 
such that the following holds: 
For any $f \in C(\partial\widehat K)$ that is edgewise a polynomial of degree $p$ there
is a lifting ${\mathcal L} f \in \Pi_p(\widehat K)$ with the additional property 
$$
\|{\mathcal L} f\|_{L^\infty(\widehat K)} + p^{-2} \|\nabla {\mathcal L}f\|_{L^\infty(\widehat K)} 
\leq C \|f\|_{L^{\infty}(\partial\widehat K)}. 
$$ 
\end{lemma}
\begin{proof}
We only illustrate the case $\widehat K = \{(x,y)\,|\, 0 < x < 1, 0 < y < 1-x\}$ of a triangle, 
following \cite[Lemma~{3.1}]{babuska-szabo-katz81}. 
After subtracting the polynomial of degree $1$ that interpolates in the three vertices, we may
assume that $f$ vanishes in the three vertices. The lifting is constructed for each edge
separately. We consider the edge $e = (0,1) \times \{0\}$ and define the lifting by 
${\mathcal L}_e f(x,y):= f(x) \frac{1-x-y}{1-x}$. Then 
$\|{\mathcal L}_e f\|_{L^\infty(\widehat K)} \leq \|f\|_{L^\infty(e)}$. For the gradient estimate,
we consider only $\partial_y {\mathcal L}_e f = - f(x)/(1-x)$ and note 
$\sup_{x \in (0,1)} |f(x)/(1-x)| \leq \|f^\prime\|_{L^\infty(0,1)} \leq 2 p^2 \|f\|_{L^\infty(0,1)}$, 
where the last estimate expresses Markov's inequality for polynomials of degree $p$ 
(see, e.g., \cite[Chap.~4, Thm.~{1.4}]{devore93} for a proof). 
\end{proof}
\begin{lemma}
\label{lemma:weighted-poincare}
Let $\omega \subset \BbbR^2$ be a bounded, open set. Let $\overline{x} \in \partial\omega$ and assume
that $\omega$ satisfies an exterior cone condition at $\overline{x}$:  There is a rotation $Q \in \BbbR^{2 \times 2}$
and a constant $c > 0$ such that the set $\overline{x} + Q K \subset \BbbR^2 \setminus \overline{\omega}$, where 
$K := \{x = (x_1,x_2) \in \BbbR^2\,|\, 0 < x_2 < c |x_1|\}$. Then there exists $C > 0$ depending solely on $c$ such that 
$$
\left\|\frac{1}{\operatorname*{dist}(\cdot,\overline{x})} u\right\|_{L^2(\omega)} \leq C \|\nabla u\|_{L^2(\omega)} 
\qquad \forall u \in H^1_0(\omega). 
$$
\end{lemma}
\begin{proof}
The desired estimate is scale invariant. We therefore assume $\operatorname*{diam} \omega \leq 1$. 
Furthermore, we assume $\overline{x} = 0$ and $Q = \operatorname*{I}$. We note that 
$\widetilde \omega:= B_1(0) \setminus K $ is a Lipschitz domain. Let $\widetilde u$ be the zero extension of $u$ to $\BbbR^2$. 
Standard estimates then provide 
$$
\left\|\frac{1}{\operatorname*{dist}(\cdot,\partial\widetilde \omega)} \widetilde u\right\|_{L^2(\widetilde\omega)} 
\leq C \|\nabla u\|_{L^2(\widetilde\omega)}. 
$$
Since $0 \in \partial\widetilde\omega$, the result follows. 
\end{proof}
\begin{lemma}
\label{lemma:inverse-estimate}
Let $h_x$, $h_y \in (0,1]$. Let $S_h:= (0,h_x) \times (0,h_y)$ and 
$T_h:= \{(x,y)\,|\, 0 < x < h_x, 0 < y < h_y (1-x/h_x)\}$. 
Then there exists $C > 0$ such that for all $p \in \BbbN$ and all $\pi \in {\Pi}_p$ 
\begin{align*}
\|\pi \|_{L^\infty(S_h)} &\leq 
C p (h_y/h_x)^{1/2} \|\partial_y \pi \|_{L^2(S_h)} + \|\pi(\cdot,0)\|_{L^\infty(0,h_x)}, \\
\|\pi \|_{L^\infty(T_h)} &\leq 
C p (h_y/h_x)^{1/2} \|\partial_y \pi \|_{L^2(T_h)} + \|\pi(\cdot,0)\|_{L^\infty(0,h_x)}. 
\end{align*}
\end{lemma}
\begin{proof}
We only prove the second estimate for the triangle $T_h$. 
{}From the representation $\pi(x,y) = \pi(x,0) + \int_{0}^y \partial_y \pi(x,t)\,dt$, we get 
for $(x,y) \in T_h$ 
$$
|\pi(x,y)|^2 \lesssim |\pi(x,0)|^2 + \left|\int_0^y \partial_y \pi(x,t)\,dt\right|^2 
\lesssim |\pi(x,0)|^2 + h_y \int_{0}^{h_y(1-x/h_x)} |\partial_y \pi(x,t)|^2\,dt. 
$$
Since $x \mapsto \int_0^{h_y(1-x/h_x)} |\partial_y \pi(x,t)|^2\,dt$ is a polynomial of degree $2p+1$, the 
polynomial inverse estimate $\|z\|_{L^\infty(0,h_x)} \leq C p^{2} h_x^{-1}\|z\|_{L^1(0,h_x)}$ for polynomials 
$z $ of degree $p$ (see, e.g.,  \cite[Chap.~4, Thm.~{2.6}]{devore93} for a proof) yields the desired bound. 
\end{proof}
\bibliographystyle{elsart-num-sort}
\bibliography{nummech}
\end{document}